\documentclass[11pt]{amsart}
\usepackage{amscd,amssymb,graphicx, hyperref}
\usepackage{color}

\begin{document}

\newtheorem{thm}{Theorem}
\newtheorem*{thm*}{Theorem}
\newtheorem{cor}[thm]{Corollary}
\newtheorem{lemma}[thm]{Lemma}
\newtheorem{prop}[thm]{Proposition}
\theoremstyle{definition}
\newtheorem{quest}[thm]{Question}
\newtheorem{ex}[thm]{Example}

\newcommand{\linnum}{\stepcounter{thm}\tag{\thethm}}

\newcommand{\bN}{\mathbb N}
\newcommand{\bR}{\mathbb R}
\newcommand{\cA}{\mathcal{A}}
\newcommand{\cC}{\mathcal{C}}
\newcommand{\cD}{\mathcal{D}}
\newcommand{\cE}{\mathcal{E}}
\newcommand{\cN}{\mathcal{N}}
\newcommand{\cO}{\mathcal{O}}
\newcommand{\cR}{\mathcal{R}}
\newcommand{\cS}{\mathcal{S}}
\newcommand{\cU}{\mathcal{U}}
\newcommand{\cV}{\mathcal{V}}
\newcommand{\cW}{\mathcal{W}}
\newcommand{\cY}{\mathcal{Y}}
\newcommand{\tx}{\tilde{x}}
\newcommand{\tv}{\tilde{v}}
\newcommand{\wtD}{{\widetilde{D}}}
\newcommand{\wtgamma}{\widetilde{\gamma}}
\newcommand{\wtphi}{\widetilde{\phi}}
\newcommand{\wtU}{{\widetilde{U}}}
\newcommand{\wtV}{{\widetilde{V}}}
\newcommand{\za}{\alpha}
\newcommand{\zb}{\beta}
\newcommand{\zf}{\phi}
\newcommand{\zg}{\gamma}
\newcommand{\zd}{\delta}
\newcommand{\zh}{\eta}
\newcommand{\zj}{\psi}
\newcommand{\zl}{\lambda}
\newcommand{\zm}{\mu}
\newcommand{\zn}{\nu}
\newcommand{\zp}{\pi}
\newcommand{\zr}{\rho}
\newcommand{\zs}{\sigma}
\newcommand{\zS}{\Sigma}
\newcommand{\zt}{\tau}
\newcommand{\zv}{\varphi}
\newcommand{\zJ}{\Psi}
\newcommand{\zG}{\Gamma}
\newcommand{\zL}{\Lambda}
\newcommand{\lcm}{\mbox{lcm}}
\newcommand{\co}{\colon\thinspace}
\newcommand{\SR}{S_{\mathcal{R}}}
\newcommand{\SRp}{S_{\mathcal{R'}}}
\newcommand{\expm}{\varphi}
\newcommand{\subm}{\sigma_{\mathcal{R}}}
\newcommand{\defn}{\emph}
\newcommand{\Chat}{\widehat{\mathbb C}}
\newcommand{\Otilde}{\widetilde{\mathcal{O}}}

\newcommand{\gap}{\vspace{5pt}}      
\newcommand{\mtwo}[4]                
{\mbox{$\left(\begin{array}{cc}      
#1 & #2 \\
#3 & #4
\end{array}
\right)$}}

\newcommand{\pf}{\noindent {\bf Proof: }}
\newcommand{\id}{\mbox{\rm id}}
\newcommand{\interior}{\mbox{int}}   
\newcommand{\diam}{\mbox{\rm diam}}
\newcommand{\dist}{\mbox{\rm dist}}
\newcommand{\per}{\mbox{\rm per}}
\newcommand{\qedspecial}[1]{\nopagebreak \begin{flushright}
        \rule{2mm}{2.5mm}{\bf #1} \end{flushright}}

\newcommand{\nosubsections}{\renewcommand{\thethm}{\thesection.\arabic{thm}}
           \setcounter{thm}{0}}

\newcommand{\bdry}{\partial}

\title{Expansion properties for finite subdivision rules I}

\author{W. J. Floyd}
\address{Department of Mathematics\\ Virginia Tech\\
Blacksburg, VA 24061\\ U.S.A.}
\email{floyd@math.vt.edu}
\urladdr{http://www.math.vt.edu/people/floyd}

\author{W. R. Parry}
\address{Department of Mathematics\\ Eastern Michigan University\\
Ypsilanti, MI 48197\\ U.S.A.}
\email{walter.parry@emich.edu}

\author{K. M. Pilgrim}
\address{Department of Mathematics\\ Indiana University\\
Bloomington, IN 47405\\ U.S.A.}
\email{pilgrim@indiana.edu}

\date\today
\maketitle

\begin{abstract}
Among Thurston maps (orientation-preserving, postcritically finite
branched coverings of the 2-sphere to itself), those that arise as
subdivision maps of a finite subdivision rule form a special family.
For such maps, we investigate relationships between various notions of
expansion---combinatorial, dynamical, algebraic, and
coarse-geometric.
\end{abstract}

\tableofcontents

\centerline{\dedicatory{Dedicated to the memory of William P. Thurston}}

\section{Introduction}\label{sec:intro}

This is the first in a series of papers devoted to expansion
properties of Thurston maps and of finite subdivision rules (fsr's).

It is a robust principle in discrete-time dynamical systems that
expanding maps are determined, up to topological conjugacy, by a
finite amount of combinatorial data.  For example, the topological
conjugacy class of a smooth expanding map $g: S^1 \to S^1$ is
determined by its degree.  Also, in many settings, homotopy or other
similar classes of maps which contain smooth expanding maps can be
identified combinatorially, and shown to contain smooth (indeed,
algebraic) models.  For example, a continuous map $f: S^1 \to S^1$ is
homotopic to an expanding map $g$ if and only if $|\deg(f)|>1$; in
this case one may take $g$ as $x \mapsto \deg(f)x \bmod 1$.  Similar
results hold for torus maps $f: T^2 \to T^2$ and indeed much more
general results hold, cf.  \cite[Theorem 3]{S}.  The expanding maps
$g$ are necessarily covering maps.

In one real dimension, multimodal interval maps $f: I \to I$, studied
in detail by Milnor and Thurston \cite{MT}, provide a natural class of
noninvertible maps with local branching.  Their {\em kneading theory}
provides a rich set of combinatorial invariants.  In this setting
also, one may formulate a notion of combinatorial classes of maps, and
one may characterize those classes containing expanding maps as well.
The presence of local branch (turning) points means that the notion of
expanding needs to be properly formulated.  Up to topological
conjugacy, the subclass of such maps $g$ with eventually periodic
turning points---called {\em postcritically finite} maps---is
countable.  A general result then implies that each postcritically
finite expanding map $g$ is topologically conjugate to an algebraic
model---a piecewise linear map where the absolute value of the slope
is constant.

In two real dimensions, Thurston \cite{DH}, \cite{Th} generalized the
kneading theory to orientation-preserving postcritically finite
branched coverings $f: (S^2,P) \to (S^2,P)$; here, $P$ is the set of
forward orbits of branch points, and is required to be finite.  These
maps $f$ are now often called {\em Thurston maps} for brevity.  The
central result---sometimes referred to as the Fundamental Theorem of
Complex Dynamics---is (i) a combinatorial characterization of rational
functions $g$ among Thurston maps, and (ii) a rigidity result, which
asserts that the holomorphic conjugacy class of a rational Thurston
map $g$ is (with a well-understood class of exceptions) uniquely
determined by its combinatorics.  The second of these results says
that in addition to the topological dynamics, the conformal structure
on $(S^2,P)$ left invariant by $g$ is encoded in the combinatorics of
$g$.  At around the same time, Cannon \cite{cannon:rmt} was
considering the problem of the hyperbolization of certain
three-manifolds.  His approach was to formulate what is now called
{\em Cannon's conjecture}: a Gromov hyperbolic group $G$ with $\bdry
G$ homeomorphic to $S^2$ acts geometrically on hyperbolic three-space.
Bonk and Kleiner \cite{BK} reinterpreted this in the following way:
the natural quasiconformal structure on $\bdry G$ determined by the
family of so-called visual metrics is, conjecturally, equivalent to
that of the Euclidean metric.

The presence of a natural combinatorial structure on $\bdry G$
generated much interest in the problem of characterizing its
quasisymmetry class through asymptotic combinatorial means.  These
structures were highly reminiscent of cellular {\em Markov
partitions} occuring in dynamics, and these began to be studied in
their own right \cite{fsr, ratsub, expi, expii, subrat}.  Now called
{\em finite subdivision rules}, these often take the form of
a Thurston map whose inverse branches refine some tiling of $S^2$; see
\S 2.  Such Thurston maps are, a priori, special among general
Thurston maps.  Along the way, various notions of topological
expansion and of combinatorial expansion were formulated for Thurston
maps, and relationships among these notions were established \cite{BM,
fsr, virt, HP2}. In independent remarkable developments, it was shown
that an expanding rational Thurston map and, more generally, an
expanding Thurston map, has the property that some iterate is the
subdivision map of an fsr  \cite{subrat, BM}.

It became recognized that bounded valence finite subdivision
rules---those for which there is an upper bound on the valence of
$0$-cells independent of the number of levels of subdivision---behave
more regularly than their unbounded valence counterparts.  In this
context, an fsr has unbounded valence if and only if the Thurston map
defining its subdivision rule has a periodic branch point.  For
example, the unbounded valence barycentric subdivision rule admits two
natural subdivision maps.  One is a rational map for which there
exists a fixed critical point and for which the diameters of the tiles
at level $n$ do not tend to zero as $n \to \infty$.  The other is a
piecewise affine map, for which the diameters of the tiles do tend to
zero.  This distinction led these authors to the realization that the
formulation of a reasonable notion of combinatorial expansion for
Thurston maps with periodic branch points is subtle.

In this first work, we suppose $f: S^2 \to S^2$ is a Thurston map
which is the subdivision map of an fsr, and we investigate properties
of $f$, focusing on themes related to notions of expansion.  We
discuss combinatorial, group-theoretic, dynamical, and coarse
geometric properties.

Here is an outline and a summary of our main results. A heads-up to
the reader: in smooth dynamics, it is common to focus on forward
iterates and expansion.  But in our setting, it is often more
convenient to consider backward iterates and contraction. So below,
one should read `expanding' as `forward iterates are expanding' and
`contracting' as `backward iterates are contracting'. For example, the
map $x \mapsto 2x$ modulo $1$ is both expanding and contracting.

Throughout, we assume $f$ is the subdivision map of a finite
subdivision rule $\cR$ .

\S 2 introduces basic terminology related to Thurston maps and fsr's. 

\S 3 formulates several distinct notions of combinatorial expansion
for fsr's, and systematically studies the subtleties mentioned above.
These notions are local and are defined in terms of tile types.  We
discuss their relationships, and give examples and
counterexamples. While many of these notions coincide in the case of
bounded valence, in general, this is not the case. We define a
particular form of combinatorial expansion, here denoted {\em Property
(CombExp)}, and single it out as the one which is best suited for our
purpose of characterizing various equivalent notions.

 \S 4 quickly recalls the construction of group-theoretic invariants
associated to $f$, focusing on the associated {\em virtual
endomorphism} $\phi_f$ of the orbifold fundamental group.

 \S 5 introduces the notion of a {\em combinatorially contracting}
(CombContr) fsr.  Unlike the local notions introduced in \S 3,
combinatorially contracting is a global notion.  Our first main
result, Proposition \ref{prop:cmblexpn}, gives sufficient local
combinatorial conditions for an fsr to be combinatorially contracting.
Corollary \ref{cor:cmblexpn} then implies that property (CombExp)
implies (CombContr).  Examples in \S 3 show that the converse fails,
however.

 \S 6 presents results related to algebraic notions:
\begin{itemize}
\item
\noindent{\bf Theorem 6.2}: {\em Let $f\co S^2\to S^2$ be a Thurston 
map which is Thurston equivalent to the subdivision map of a finite
subdivision rule.  Then there exist
\begin{enumerate}
  \item nonnegative integers $a$ and $b$;
  \item a choice of orbifold fundamental group virtual endomorphism
$\zf\co G_p\dashrightarrow G_p$ associated to $f$;
  \item a finite generating set for the orbifold fundamental group
$G_p$ with associated length function $\left\|\cdot \right\|$
\end{enumerate}
such that
  \begin{equation*}
\left\|\zf^n(g)\right\|\le a \left\|g\right\|+bn
  \end{equation*}
for every nonnegative integer $n$ and $g\in \text{dom}(\zf^n)$.  If
the subdivision map takes some tile of level 1 to the tile of level 0
which contains it, then we may take $b=0$.  If there is only one tile
of level 0, then we may take $a=1$.}\medskip

\item \noindent{\bf Theorem 6.5}: {\em $\cR$ is contracting if and
only if $\phi_f$ is contracting.} 
\end{itemize}

Theorem~\ref{thm:expsep} shows that if $\cR$ has bounded valence, then
property (CombExp) is equivalent to property (M0Comb), mesh
approaching 0 combinatorially.  Corollary~\ref{cor:cmblexpn} shows
that property (CombExp) implies that $\cR$ is contracting.
Combining these results with Theorem~\ref{thm:cntrn} implies that if
$\cR$ has bounded valence and the mesh of $\cR$ approaches 0
combinatorially, then $\zf_f$ is contracting.  This is the main result
of \cite{virt}.

\begin{itemize}
\item {\bf Theorem 6.7:} {\em If $f$ is combinatorially equivalent to
a map $g$ that, outside a neighborhood of periodic branch points, is
expanding with respect to some length metric, then $\cR$ is
contracting.}
\end{itemize}

In \S\S 7 and 8, we turn our attention to coarse (large-scale)
geometric notions. Associated to $f$ are two natural infinite
1-complexes.  The {\em fat path subdivision graph} joins a tile to
its parent by a vertical edge and to its edge-adjacent neighbors by a
horizontal edge. The {\em selfsimilarity complex} is built from the
iterated monodromy of $f$ under the action of the orbifold fundamental
group.\medskip

\noindent {\bf Theorem 7.4:} {\em If $\cR$ is contracting, then its
fat path subdivision graph is Gromov hyperbolic.}\medskip

For the proof, we use some general results of Rushton \cite{R}.  These
lead to a condition on $\cR$ which is necessary and sufficient for the
fat path subdivision graph to be Gromov hyperbolic.  This necessary
and sufficient condition is very similar to but strictly weaker than
the condition that $\cR$ is contracting.\medskip

\noindent{\bf Theorem 8.1:} {\em The fat path subdivision complex and 
(any) selfsimilarity complex are quasi-isometric.}\medskip

\S 9 begins with a discussion of dynamical plane obstructions for an
fsr to be contracting.  The most natural candidate is a so-called {\em
Levy obstruction}.  It is easy to see that the existence of Levy
obstructions implies that $\phi_f$ is not contracting, and hence that
$\cR$ is not contracting.  However, we do not know if Levy
obstructions are the only obstructions. We conclude with some related
open questions.  \\

\subsection{Acknowledgement} We express heartfelt gratitude to Jim
Cannon, who helped us in many ways.  Without him this paper would not
exist. K. Pilgrim was supported by Simons grant 245269.

\section{Basic definitions}\label{sec:defns}\nosubsections

A continuous map $f\co S^2 \to S^2$ is a \defn{branched covering} if
$f$ is orientation preserving and for each $x\in S^2$ there are local
charts about $x$ and $f(x)$ (sending $x$ and $f(x)$ to $0$) in which
$f$ becomes the map $z\mapsto z^{\deg_x(f)}$ for some positive integer
$\deg_x(f)$.  The \defn{critical set} of $f$ is $\Omega_f = \{x:
\deg_x(f) > 1\}$ (this is the set of points at which $f$ is not
locally injective) and the \defn{postcritical set} is $P_f = \cup_{n >
0} f^n(\Omega_f)$. The mapping $f$ is called \defn{postcritically
finite} (or \defn{critically finite}) or a \defn{Thurston map} if
$P_f$ is finite.  Two Thurston maps $f$ and $g$ are
\defn{combinatorially equivalent} or \defn{Thurston equivalent} if
there are orientation-preserving homeomorphisms $h_0, h_1\co (S^2,P_f)
\to (S^2,P_g)$ such that $h_0 \circ f = g\circ h_1$ and $h_0$ and
$h_1$ are isotopic rel $P_f$.  \gap

A \defn{preimage} of a connected set $U$ under $f$ is a connected
component of $f^{-1}(U)$.  Let $\cU_0$ be a finite open cover of $S^2$
by connected sets.  Inductively define finite open covers $\cU_n$ of
$S^2$ by setting $\cU_{n+1}$ to be the cover whose elements are
preimages of elements of $\cU_n$. Following Ha\"issinsky-Pilgrim
\cite{HP1}, a Thurston map $f$ is said to \defn{satisfy Axiom
[Expansion] with respect to $\cU_0$} if for every open cover $\cY$ of
$S^2$, there exists $N \in \bN$ such that for all $n \geq N$, every
element of $\cU_n$ is contained in an element of $\cY$. We say that
$f$ \defn{satisfies Axiom [Expansion]} if it satisfies Axiom
[Expansion] with respect to some such cover $\cU_0$.  Equipping $S^2$
with e.g. the spherical metric, letting $\epsilon>0$ be arbitrary, and
taking $\cY$ to be a cover by $\epsilon$ balls shows that Axiom
[Expansion] is equivalent to the assertion that there exists a finite
cover by connected open sets whose iterated preimages under $f^ n$
have diameters tending to zero uniformly as $n\to \infty$. \gap

A \defn{finite subdivision rule} $\cR$ consists of a finite
2-dimensional CW complex $\SR$ which is the union of its closed
2-cells, a subdivision $\cR(\SR)$ of $\SR$, and a continuous cellular
map $f\co \cR(\SR) \to \SR$ whose restriction to each open cell is a
homeomorphism.  Furthermore, for each closed 2-cell $\tilde{t}$ of
$\SR$ there are i) a cell structure $t$ on the 2-disk $D^2$ ($t$ is
called the \defn{tile type} of $\tilde{t}$) such that the 1-skeleton
of $t$ is $\partial D^2$ and $t$ has at least three vertices and ii) a
continuous surjection $\psi_t\co t \to \tilde{t}$ (called the
\defn{characteristic map} of $\tilde{t}$) whose restriction to each
open cell is a homeomorphism. If $\cR$ is a finite subdivision rule, a
2-dimensional CW complex $X$ is an $\cR$-complex if it is the union of
its closed 2-cells, and there is a continuous cellular map $g\co X\to
\SR$ whose restriction to each open cell is a homeomorphism.  In this
case, for each positive integer $n$ there is a subdivision $\cR^n(X)$
of $X$ with associated map $f^n\circ g$.

When the underlying space of $S_\cR$ is homeomorphic to the 2-sphere
$S^2$ and $f$ is orientation preserving (for concreteness, we consider
only this case), $f$ is a Thurston map. Since $\SR$ is an
$\cR$-complex, we get a recursive sequence $\{\cR^n(\SR)\}$ of
subdivsions of $\SR$.  For each positive integer $n$, $f$ is a
cellular map from the $n^{\textrm{th}}$ to the $(n-1)^{\textrm{st}}$
subdivision.  Thus, we may speak of \defn{tiles} (which are closed
2-cells), \defn{edges}, and \defn{vertices} at \defn{level} $n$.  The
subdivision rule $\cR$ has \defn{bounded valence} if there is a
uniform upper bound on the valence of vertices of $\{\cR^n(\SR)\}$.
It is important to note that formally a finite subdivision rule is
\defn{not} a combinatorial object, since the map $f$, which is part of
the data, is assumed given. In other words: as a dynamical system on
the 2-sphere, the topological conjugacy class of $f$ is well defined.

Two finite subdivision rules $\cR, \cR'$ with subdivision maps $f$,
$f'$ are \defn{weakly isomorphic} if there is a cellular isomorphism
$h: S_\cR \to S_{\cR'}$ such that $h$ is cellularly isotopic to a
cellular isomorphism $g: S_\cR \to S_{\cR'}$ such that $f'\circ g = h
\circ f$.  In the case of finite subdivision rules on the sphere,
$\cR$ weakly isomorphic to $\cR'$ implies that $f$ and $f'$ are
combinatorially equivalent as Thurston maps. The finite subdivision
rules $\cR, \cR'$ are \defn{isomorphic} provided there is a cellular
isomorphism $h: \cS_\cR \to \cS_{\cR'}$ with $f' \circ h = h \circ
f$. That is, $f$ and $f'$ are topologically conjugate as dynamical
systems via a cellular isomorphism.  \gap

\section{Expansion properties for finite subdivision rules}
\label{sec:exp}\nosubsections

We are interested here in various definitions of combinatorial
expansion for finite subdivision rules.  We begin with some
definitions of separation and expansion properties.  Let $\cR$ be a
finite subdivision rule.
\begin{enumerate}
\item[(Esub)] $\cR$ is \emph{edge subdividing} if whenever $e$ is an
edge in $\SR$, there is a positive integer $n$ such that $e$ is
properly subdivided in $\cR^n(\SR)$.
\item[(Esep)] $\cR$ is \emph{edge separating} if whenever
$t$ is a tile type of $\cR$ and $e$ and $e'$ are disjoint edges
of $t$, there is a positive integer $n$ such that no tile of
$\cR^n(t)$ contains subedges of $e$ and of $e'$.
\item[(VEsep)] $\cR$ is \emph{vertex-edge separating} if whenever $t$
is a tile type, $u$ is a vertex of $t$, and $e$ is an edge of $t$ not
containing $u$, there exists a positive integer $n$ such that no tile
of $\cR^n(t)$ contains both $u$ and an edge contained in $e$.
\item[(Vsep)] $\cR$ is \emph{vertex separating} if
whenever $t$ is a tile type and $u$ and $v$ are distinct vertices of $t$,
there exists a positive integer $n$ such that no tile of $\cR^n(t)$
contains both $u$ and $v$.
\item[(Sexp)] $\cR$ is \emph{subcomplex expanding} if whenever $X$ is
an $\cR$-complex and $M > 0$, then there exists a natural number $N$
such that if $Y$ and $Z$ are disjoint subcomplexes (not necessarily
unions of tiles) of $X$ then the edge-path distance in $\cR^n(X)$
between $Y$ and $Z$ is greater than $M$ if $n\ge N$.
\item[(M0)] $\cR$ has \emph{mesh approaching 0} if given an open cover
$\cU$ of $S_\cR$ there exists a positive integer $n$ such that every
tile of $\cR^n(S_\cR)$ is contained in an element of $\cU$.
\item[(M0comb)] $\cR$ has \emph{mesh approaching 0 combinatorially} if
it is edge subdividing and edge separating.  There is a finite
algorithm which detects if $\cR$ satisfies (M0comb); see \cite[Theorem
6.1]{virt}.
\item[(M0weak)] $\cR$ has \emph{mesh approaching 0 weakly} if $\cR$ is
weakly isomorphic to a finite subdivision rule $\cR'$ such that $\cR'$
satisfies (M0).  
\item[(CombExp)] $\cR$ is \emph{combinatorially expanding} if it is
edge separating, vertex-edge separating, and vertex separating.  The
argument given for \cite[Theorem 6.1]{virt} adapts immediately to
yield a finite algorithm which detects if $\cR$ satisfies (CombExp);
the bound on the level of subdivisions needed to be considered is
indeed the same as for the check of (M0comb); see Theorem
\ref{thm-complexity} below.
\end{enumerate}

We are interested in the relationships between these definitions.
\begin{itemize}
\item The first four---(Esub), (Esep), (VEsep) and (Vsep)---of these
nine properties might be called ``primary'' because Theorem 3.5 below
will show that the last five can be defined in terms of these.
\item Example 3.1 shows that (Vsep) does not follow from the other
three primary properties.
\item Example 3.2 shows that (VEsep) does not follow from the other
three primary properties.
\item Example 3.3 shows that (Esep) does not follow from the other
three primary properties.
\item It is easy to see that (Vsep) implies (Esub) (See Theorem 3.4.),
so Esub follows from the other three primary properties.
\end{itemize}

We next give several examples. In each example, the subdivision 
complex is a 2-sphere and the subdivision map is a Thurston map.
Each example is accompanied by a figure giving the
subdivisions of the tile types. For each example, the edge
labelings and their cyclic order determine the tile types of all
of the tiles, so the tile types are not specified.

\begin{ex}\label{ex:ex1} This finite subdivision rule shows
that (Vsep) does not follow from the other three primary properties.
Let $f(z)=1-1/z^2$. The postcritical set is $\{0,1,\infty\}$.  We
equip the Riemann sphere with a cell structure for which the
1-skeleton is the real axis together with $\infty$ and the set of
0-cells is the postcritical set.  The associated finite subdivision
rule $\cR_1$ has two tile types, and the subdivision complex
$S_{\cR_1}$ is a triangular pillowcase. The subdivisions of the tile
types are given in Figure~\ref{fig:ex1sub}. Since both tile types are
triangles, a tile cannot have disjoint edges and so (Esep) is
satisfied vacuously. Since the edges of each tile type are properly
subdivided in the third subdivision, $\cR_0$ also satisfies (Esub) and
hence (M0comb).  The vertex 1, which is contained in edges $a$ and
$b$, is separated from the edge opposite it in the first subdivision.
So vertex $\infty$ is separated from the edge opposite it in the
second subdivision.  Vertex 0 is separated from the edge opposite it
in the third subdivision.  So $\cR_1$ satisfies (VEsep).  However,
(Vsep) is not satisfied because the subdivision of each tile type
contains a tile of the same type with vertices 0, 1 and $\infty$.   
Note from Figure \ref{fig:ex1julia} that the immediate basins of
postcritical points have closures that meet.

\begin{figure}
\centerline{\includegraphics[height=2.5in]{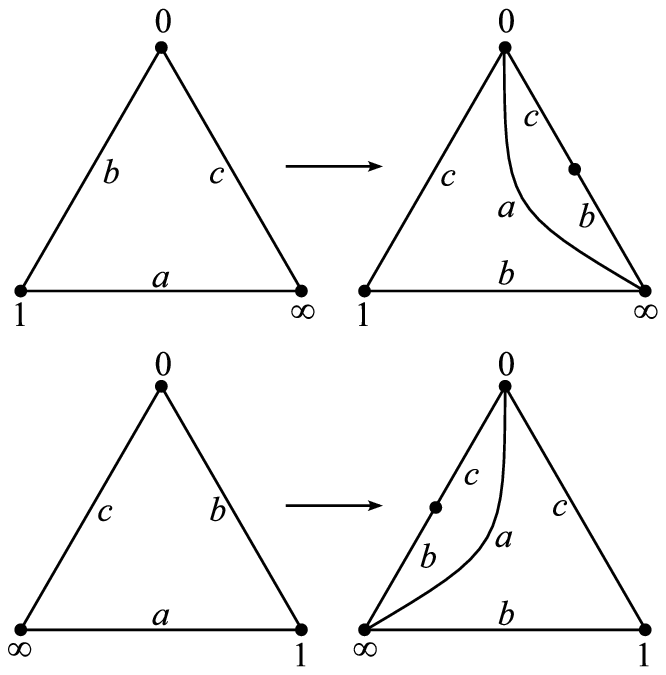}} 
\caption{Example 3.1: The subdivisions of the tile types for $\cR_1$} 
\label{fig:ex1sub}

\includegraphics[height=2.5in]{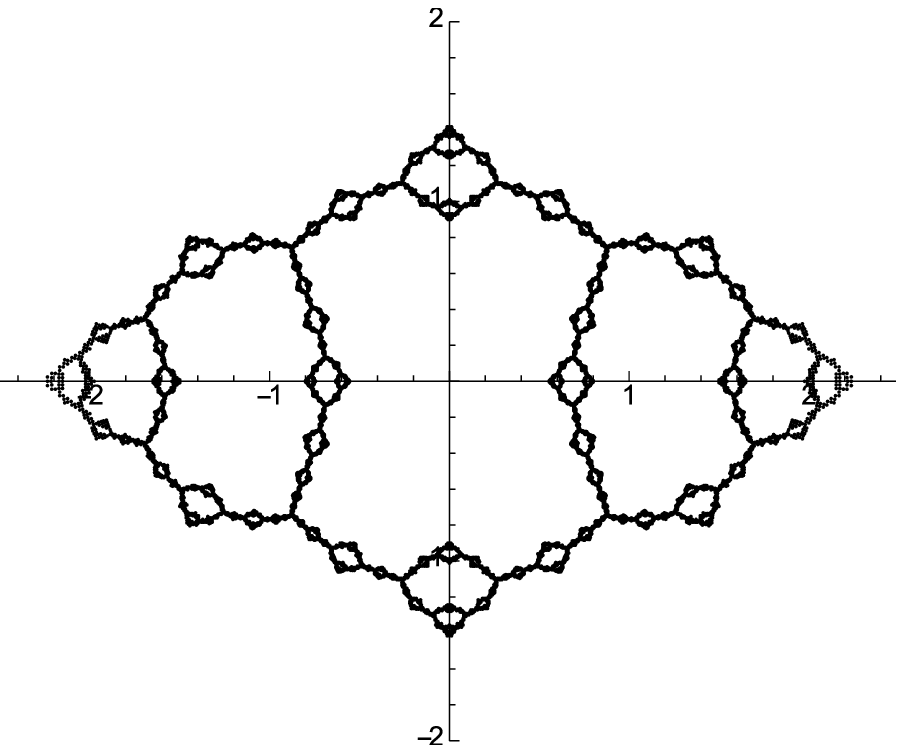}
\caption{Example 3.1: The
Julia set for the subdivision map of $\cR_1$.}
\label{fig:ex1julia}
\end{figure}

\end{ex}

\begin{ex}\label{ex:tr6} This finite subdivision rule $\cR_2$ shows
that (VEsep) does not follow from the other three primary
properties.  The subdivision complex is a triangular pillowcase. The
subdivisions of the two tile types are given in
Figure~\ref{fig:tr6sub}, and Figure~\ref{fig:tr6l14} (which was drawn
using CirclePack \cite{CP}) shows the first four subdivisions of the
quadrilateral $Q$ obtained by gluing the two triangles along the edge
labelled $c$.  $\cR_2$ satisfies (Esep) vacuously since the tile types
are triangles. For each tile type $t$, each edge of $t$ is properly
subdivided in $\cR_2(t)$ and so $\cR_2$ satisfies (Esub) and (M0comb).

This finite subdivision rule satisfies (Esub), (Esep), and (Vsep);
it does not satisfy (VEsep) or (Sexp). 
The subdivision map is realized by the rational map
$$f(z) = - \frac{(- i + \sqrt{3} + 8 i z -8 i z^2)^3}{96 \sqrt{3}
(-z + 5z^2 -8z^3 + 4z^4)}.$$
\end{ex}

\begin{figure}
\includegraphics[height=2.5in]{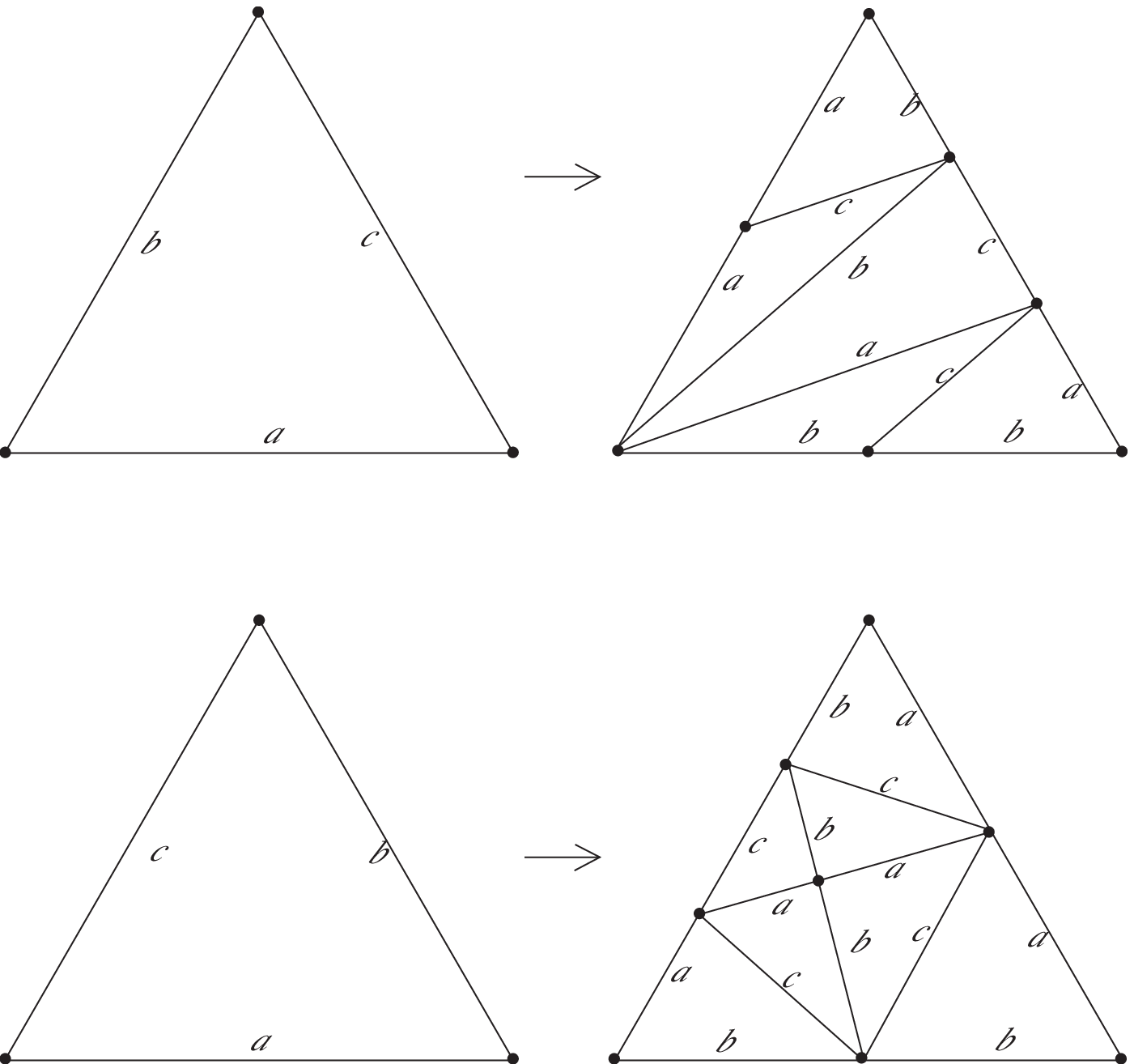} 
\caption{Example 3.2: the subdivisions of the tile types for $\cR_2$} 
\label{fig:tr6sub}

\includegraphics[height=2.5in]{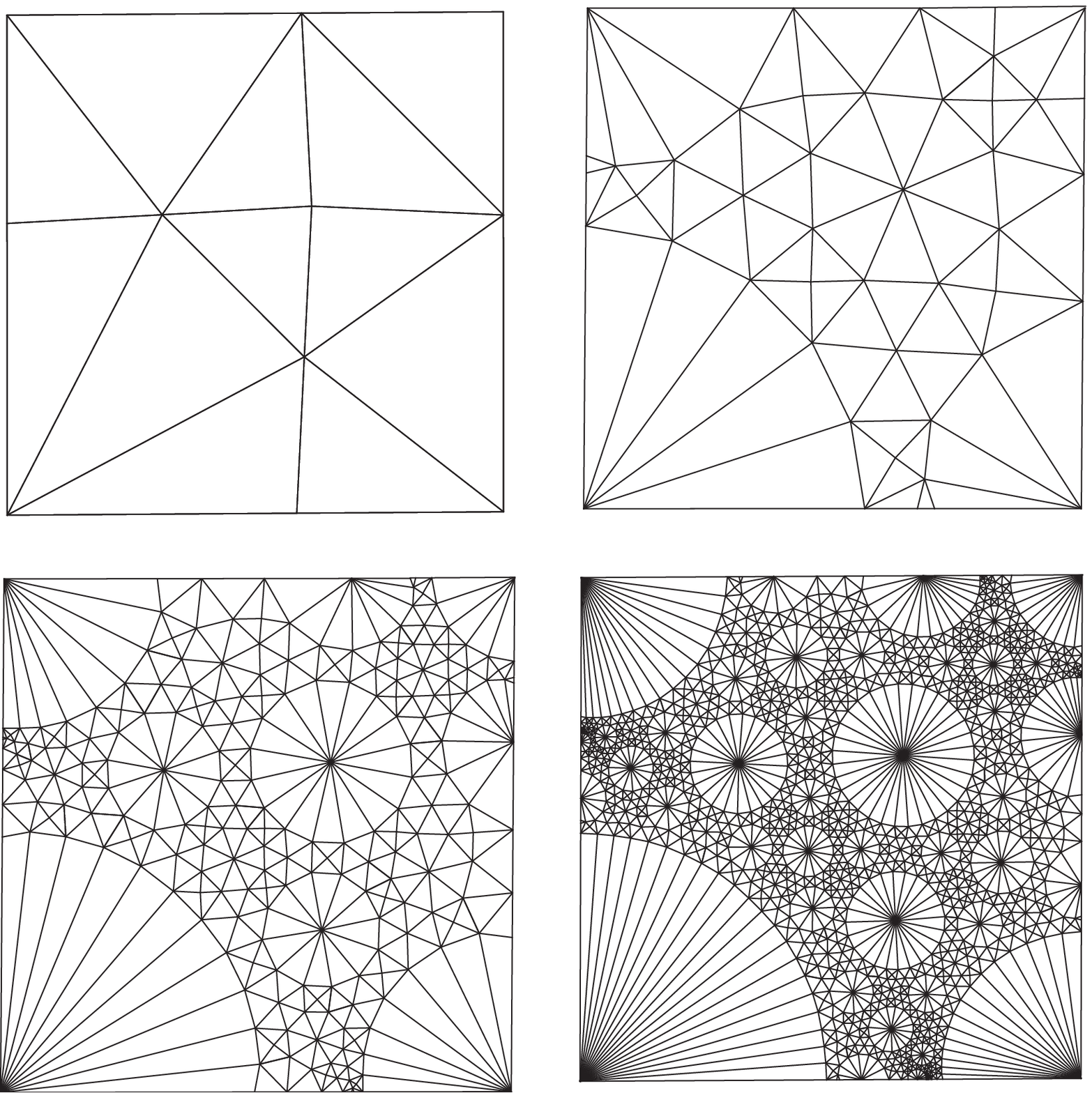} 
\caption{Example 3.2: $\cR_2^n(Q)$ for $n=1,2,3,4$}
\label{fig:tr6l14}
\end{figure}

\begin{ex}\label{ex:quad15} This finite subdivision rule $\cR_3$ shows
that (Esep) does not follow from the other three primary
properties. The subdivision rule has bounded valence and has
subdivision complex a rectangular pillowcase. The subdivisions of the
two tile types are given in Figure~\ref{fig:quad15sub}, and
Figure~\ref{fig:quad15l1-3} (which was drawn using CirclePack
\cite{CP}) shows the first three subdivisions of one of the tile
types. It is clear from the subdivisions of the tile types that it
satisfies (Esub), (Vsep), and (VEsep).  Since for any $n$ and tile
type $t$ the complex $\cR_3^n(t)$ has a tile which contains subedges
of the edges of $t$ labelled $b$ and $d$, $\cR_3$ does not satisfy
(Esep) or (Sexp). It is not realizable by a rational map; a horizontal
curve gives a Thurston obstruction.
\end{ex}

\begin{figure}
\centering
\includegraphics[height=2.5in]{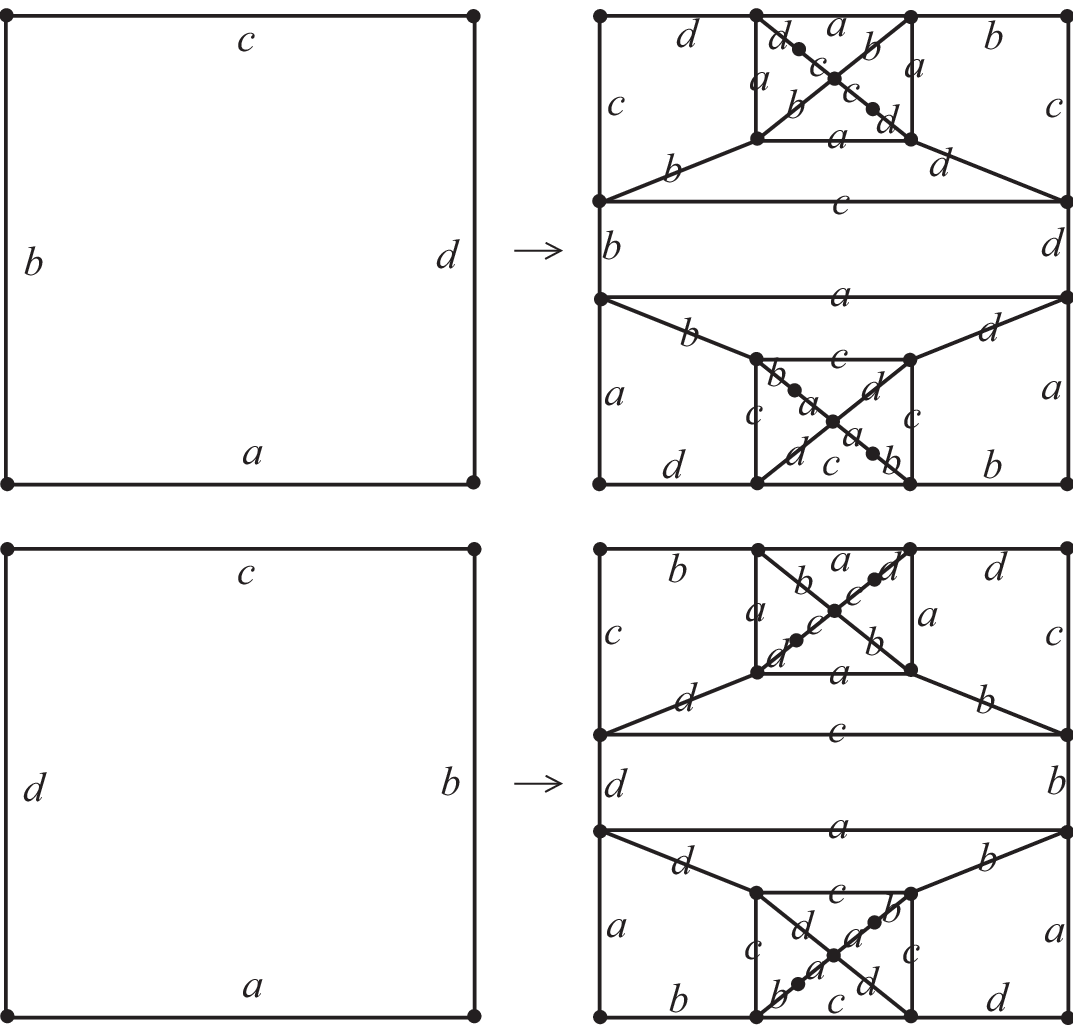}
\caption{Example 3.3: The subdivisions of the tile types for $\cR_3$} 
\label{fig:quad15sub}

\includegraphics[height=2.5in]{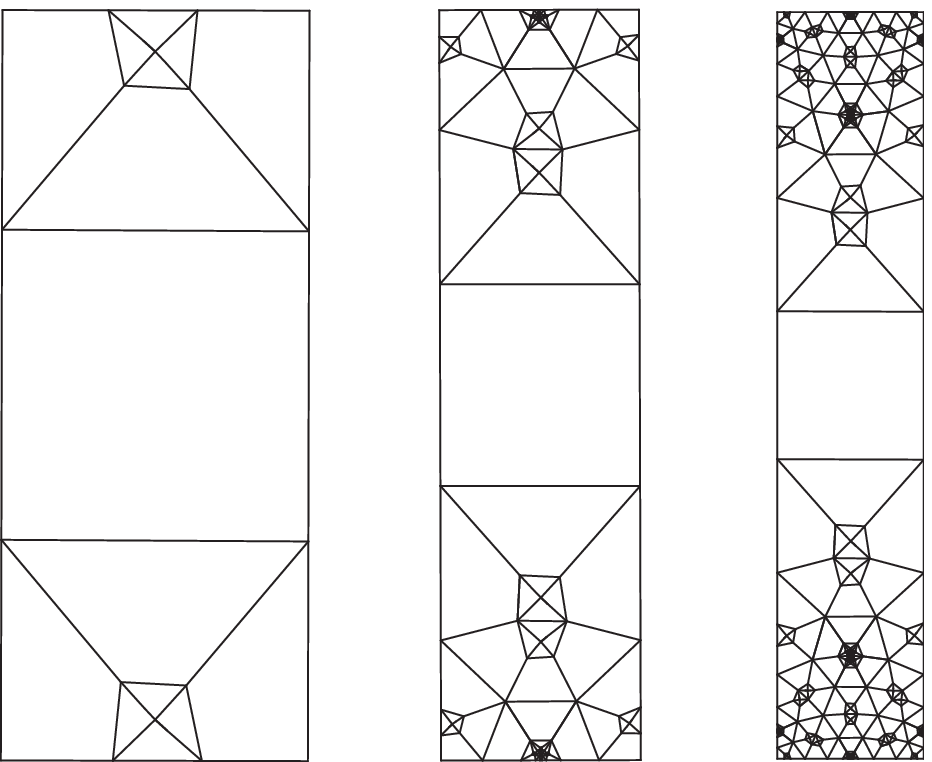}
\caption{Example 3.3: $\cR_3^n(t)$ for $n = 1,2,3$}
\label{fig:quad15l1-3}
\end{figure}

\begin{thm}\label{thm:expsep}
Let $\cR$ be a finite subdivision rule.  We have
\begin{enumerate}
  \item $\textrm{(Sexp)} \Rightarrow \textrm{(Esep)}$,
$\textrm{(Sexp)} \Rightarrow \textrm{(VEsep)}$, and $\textrm{(Sexp)}
\Rightarrow \textrm{(Vsep)} \Rightarrow \textrm{(Esub)}$. Moreover
$\textrm{(Combexp)}\Rightarrow \textrm{(Esub)}$ and so
$\textrm{(Combexp}) \Rightarrow \textrm{(M0comb)}$.
  \item $\textrm{(Combexp)} \Leftrightarrow \textrm{(Sexp)}$.
  \item $\textrm{(M0weak)} \Leftrightarrow \textrm{(CombExp)}$.  If in
addition $\cR$ has bounded valence, then $\textrm{(M0comb)}
\Leftrightarrow \textrm{(M0weak)} \Leftrightarrow \textrm{(CombExp)}$.
\end{enumerate}
\end{thm}
\begin{proof}
(1) is clear.

The reverse implication in (2) follows from (1). The forward
implication in (2) is similar to \cite[Lemma 2.7]{expi}.  We begin by
assuming that $X$ is a tile type of $\cR$ and that $Y$ and $Z$ both
consist of either one vertex or one edge.  By repeatedly applying what
it means for $\cR$ to be combinatorially expanding and using the fact
that there are only finitely many choices for $X$, $Y$, and $Z$, we
find that there exists a positive integer $N_0$ for which the
following three statements hold.  If $Y$ and $Z$ are distinct
vertices, then no tile of $\cR^{N_0}(X)$ contains both $Y$ and $Z$.
If $Y$ is a vertex not in the edge $Z$, then no tile of $\cR^{N_0}(X)$
contains $Y$ and an edge in $Z$.  If $Y$ and $Z$ are disjoint edges,
then no tile of $\cR^{N_0}(X)$ contains an edge in $Y$ and an edge in
$Z$.  Moreover, the same statements hold for the same integer $N_0$ if
$X$ is any tile in any $\cR$-complex and if $Y$ and $Z$ both consist
of either one vertex or one edge of $X$.

Now let $Y$ and $Z$ be any disjoint subcomplexes of any $\cR$-complex
$X$.  We next prove by contradiction that no tile of $\cR^{3N_0}(X)$
meets both $Y$ and $Z$.  To prove this, suppose that $t_3$ is a tile
of $\cR^{3N_0}(X)$ which meets both $Y$ and $Z$.  Let $y\in Y\cap t_3$
and $z\in Z\cap t_3$.  Let $t_i$ be the tile of $\cR^{iN_0}(X)$ which
contains $t_{i+1}$ for $i=2$, then $i=1$, and then $i=0$.  Because $y$
and $z$ are both contained in $t_3$ and the subcomplexes $Y$ and $Z$
are disjoint, neither $Y$ nor $Z$ contains $t_3$.  Hence neither $Y$
nor $Z$ contains $t_i$ for $i\in \{0,1,2,3\}$.  So neither $y$ nor $z$
is in the open 2-cell of $t_i$ for $i\in \{0,1,2,3\}$.  Now we show
that either $y$ or $z$ is a vertex of $t_1$.  If not, then because $y$
and $z$ are not in the open 2-cell of $t_1$, both $y$ and $z$ are
contained in open edges of $t_1$.  Because $y$ and $z$ are not in the
open 2-cell of $t_0$, these edges of $t_1$ are contained in edges of
$t_0$.  Because $Y$ is a subcomplex of $X$, the open edge of $t_0$
which contains $y$ is contained in $Y$.  Similarly, the open edge of
$t_0$ which contains $z$ is contained in $Z$.  Because $Y$ and $Z$ are
disjoint, so are these edges of $t_0$.  This violates the edge
separating property provided by the choice of $N_0$.  So either $y$ or
$z$ is a vertex of $t_1$.  The same reasoning now shows that both $y$
and $z$ are vertices of $t_2$.  We finally obtain a contradiction to
the vertex separating property provided by the choice of $N_0$.  Thus
no tile of $\cR^{3N_0}(X)$ meets both $Y$ and $Z$.

The previous paragraph implies that the star of $Y$ in $\cR^{3N_0}(X)$
is disjoint from $Z$.  Inductively, the $m$th star of $Y$ in
$\cR^{3mN_0}(X)$ is disjoint from $Z$ for every positive integer $m$.
So if $m\ge M$, then the edge-path distance in $\cR^{3mN_0}(X)$
between $Y$ and $Z$ is greater than $M$.  Thus we may let $N$ be any
integer greater than $3MN_0$. This proves (2).

For (3), note that by the first statement of the theorem,
$\textrm{(CombExp)} \Rightarrow \textrm{(M0comb)}$ for any finite
subdivision rule.  The equivalence of (M0comb) and (M0weak) for a
bounded valence finite subdivision rule is \cite[Theorem 2.3]{fsr}.
It is clear that for any finite subdivision rule, (M0weak) implies
(CombExp). If a finite subdivision rule satisfies (CombExp), then by
definition it satisfies the conclusion of \cite[Lemma 2.1]{fsr}.  This
implies that the proofs of Lemma 2.2 and of Theorem 2.3 of \cite{fsr}
go through without the assumption of bounded valence.  Hence (CombExp)
implies (M0weak).
\end{proof}

Examples~\ref{ex:ex1} and \ref{ex:tr6} show that if $\cR$ does not
have bounded valence, then (M0comb) does not imply (Vsep) or (VEsep).
The property of (M0comb) was defined in \cite{fsr} when the authors
were primarily interested in finite subdivision rules with bounded
valence. For finite subdivision rules without bounded valence, the
property (CombExp) is a better fit for the notion ``mesh approaching
zero combinatorially''. By the second part of
Theorem~\ref{thm:expsep}, we could have defined combinatorially
expanding as being subcomplex expanding. We used the definition we
gave instead because it is easier to verify.

In \cite[Theorem 6.1]{virt} we prove that, to verify the conditions
(Esub) and (Esep) of (M0comb) for a finite subdivision rule $\cR$, it
suffices to let $n$ be the integer $kl^2$, where $k$ is the number of
tiles in $\SR$ and $l$ is the maximum number of edges of a tile type
of $\cR$. The statement and proof easily adapt to give the following.

\begin{thm}\label{thm-complexity} Let $\cR$ be a finite subdivision 
rule, let $k$ be the number of tiles in $\SR$, and let $l$ be the
maximum number of edges of one of the tile types of $\cR$. Then $\cR$
is combinatorially expanding if and only if the conditions (Esep),
(VEsep), and (Vsep) are satisfied for the integer $n=kl^2$.
\end{thm}
\begin{proof}
The proof follows the argument in the proof of \cite[Theorem
6.1]{virt}. If the conditions (Esep), (VEsep), and (Vsep) are
satisfied for $n=kl^2$, then $\cR$ is combinatorially expanding.  Now
suppose that $\cR$ is combinatorially expanding. Let $G_e$
(respectively $G_v$) be the directed graph whose vertices are triples
$(t,a_1,a_2)$ such that $t$ is a tile type and $a_1$ and $a_2$ are
disjoint edges (respectively vertices) of $t$. Let $G_{ve}$ be the
directed graph whose vertices are triples $(t,v,e)$ such that $t$ is a
tile type, $v$ is a vertex of $t$, and $e$ is an edge of $t$ which is
disjoint from $v$. For each of the graphs, there is a directed edge
from $(t,a_1,a_2)$ to $(t',a_1',a_2')$ if the subdivision $\cR(t)$
contains a tile with type $t'$ such that the 0-cell (or 1-cell)
corresponding to $a_i'$ is equal to (or contained in) $a_i$ for
$i=1,2$. Each of the three graphs has at most $kl^2$ vertices. The
subdivision rule $\cR$ does not satisfy condition (Esep) (respectively
(Vsep), respectively (VEsep)) if and only if the graph $G_e$
(respectively $G_v$, respectvely $G_{ve}$) has a directed cycle. For
each case, if there is a directed cycle there is a directed cycle of
length at most $kl^2$, the number of vertices in the graph.
\end{proof}

\section{The virtual endomorphism}\label{sec:endo}\nosubsections

This section mainly fixes definitions and notation related to the
orbifold fundamental group virtual endomorphism.  Let $f\co S^2\to
S^2$ be a Thurston map.

\emph{The weight function $\zn$.} For $y\in S^2$ let
$\zn(y)=\text{lcm}\{\deg_x(f^n):f^n(x)=y,n\in \mathbb{N}\}$.  The set
of elements $x$ in $S^2$ such that $\zn(x)>1$ is exactly the
postcritical set $P_f$ of $f$, a finite set.

\emph{The orbifold fundamental group $G_p$.} We choose a basepoint
$p\in S^2\setminus P_f$.  For every $x\in P_f$ set $\zm(x)=\zn(x)$ if
$\zn(x)\in \mathbb{Z}$ and set $\zm(x)=0$ if $\zn(x)=\infty$.  Let
$N_p$ be the normal subgroup of the fundamental group
$\zp_1(S^2\setminus P_f,p)$ generated by all elements of the form
$g^{\zm(x)}$, where $g$ is represented by a simple loop around $x\in
P_f$ based at $p$.  The orbifold fundamental group associated to $f$
is $G_p=\zp_1(S^2\setminus P_f,p)/N_p$.  Changing the basepoint $p$
leads to an isomorphic group just as for the ordinary fundamental
group.

\emph{The orbifold universal covering map $\zp$.} Let $\zp\co D\to
S^2$ be the orbifold universal covering map for $f$.  The restriction
of $\zp$ to the points of $D$ which map to $S^2\setminus P_f$ is the
ordinary covering map associated to the subgroup $N_p$ of
$\zp_1(S^2\setminus P_f,p)$.  The image of $\zp$ is $S^2\setminus
P_f^\infty$, where $P_f^\infty=\{x\in P_f:\zn(x)=\infty\}$.  The local
degree of $\zp$ at $x\in D$ is $\zn(\zp(x))$.  So the points of $D$
which map to $P_f\setminus P_f^\infty$ are branch points of $\zp$.

\emph{The orbifold fundamental group virtual endomorphism $\zf$.} We
define a group homomorphism $\zf$ from a subgroup of $G_p$ to $G_p$.
Choose $q\in S^2$ such that $f(q)=p$.  Let $\zb$ be a path from $p$ to
$q$ in $S^2\setminus P_f$.  The domain $\text{dom}(\zf)$ of $\zf$
consists of all elements $g\in G_p$ represented by loops $\zg$ based
at $p$ such that $\zg$ lifts via $f$ to a loop $f^{-1}(\zg)[q]$ based
at $q$.  Then $\zf(g)$ is the homotopy class of $\zb
f^{-1}(\zg)[q]\zb^{-1}$.  Concatenations of paths are read from left
to right.  This definition depends on the choice of $p$, $q$ and
$\zb$, but it is unique up to pre- and post-composition by inner
automorphisms.  This defines the orbifold fundamental group virtual
endomorphism $\zf\co G_p\dashrightarrow G_p$ associated to $f$.

\emph{Properties of $\zf$.} Suppose that $f$ has degree $d$.  Let
$q_1=q,\dotsc,q_d$ be all the points in $S^2$ which $f$ maps to $p$.
Let $\zg$ be a loop based at $p$.  Then the lift $f^{-1}(\zg)[q]$ of
$\zg$ to a path based at $q$ ends at one of the points
$q_1,\dotsc,q_d$.  If $\zg=f(\za)$, where $\za$ is a path from $q$ to
$q_i$ for some $i\in\{1,\ldots,d\}$, then $f^{-1}(\zg)[q]=\za$ ends at
$q_i$.  This discussion essentially implies that $\zf$ is surjective
and that $[G_p:\text{dom}(\zf)]=d$.  Extending this argument shows
that $\zf$ is level transitive in the language of Nekrashevych's book
\cite{N}.

\emph{The contraction ratio $\zr_\zf$.} Let $\left\|g\right\|$ denote
the length of $g\in G_p$ relative to some finite generating set of
$G_p$.  The contraction ratio of $\zf$ is by definition
  \begin{equation*}
\zr_\zf=\limsup_{n\to \infty}\left(\limsup_{\left\|g\right\|\to
\infty}\frac{\left\|\zf^n(g)\right\|}{\left\|g\right\|}\right)^{1/n}.
  \end{equation*}
Of course, in the inner limit we only consider $g\in
\text{dom}(\zf^n)$.  Lemma 2.11.10 of \cite{N} shows that this limit
exists and is independent of the choice of generating set.  Because
$\zf$ is level transitive, Proposition 2.11.11 of \cite{N} implies
that $\zf$ is contracting if and only if $\zr_\zf<1$.

\emph{Iterates of $\zf$.} Let $n$ be a positive integer.  We may view
$\zf^n$ in the following way.  We have $p$ and $\zb_1=\zb$.  We
inductively define $\zb_i$ to be the lift of $\zb_{i-1}$ via $f$ such
that the initial endpoint of $\zb_i$ equals the terminal endpoint of
$\zb_{i-1}$ for $i\in\{2,\ldots,n\}$.  Let $r$ be the terminal
endpoint of $\zb_n$.  Then a loop $\zg$ at $p$ represents an element
$g$ of $\text{dom}(\zf^n)$ if and only if the lift of $\zg$ to $r$ via
$f^n$ is closed.  Furthermore, $\zf^n(g)$ is gotten by traversing
$\zb_1,\dotsc,\zb_n$, then $f^{-n}(\zg)[r]$ and then
$\zb_n^{-1},\dotsc,\zb_1^{-1}$.

\emph{$G_p\cong \text{Aut}(\zp)$.} Since $N_p$ is a normal subgroup of
$\zp_1(S^2\setminus P_f,p)$, the covering map $\zp$ is normal
(regular), and its automorphism group $\text{Aut}(\zp)$ of deck
transformations is isomorphic to $G_p$.  We next explicitly identify
such an isomorphism.  Let $\widetilde{p}\in D$ be a lift of the
basepoint $p\in S^2\setminus P_f$.  Let $\zg$ be a loop in
$S^2\setminus P_f$ based at $p$.  Let $\widetilde{\zg}$ be the lift of
$\zg$ based at $\widetilde{p}$.  There exists a unique element $\zs\in
\text{Aut}(\zp)$ such that $\zs(\widetilde{p})$ is the terminal
endpoint of $\widetilde{\zg}$.  The map $\zg\mapsto \zs$ induces a
group isomorphism from $G_p$ to $\text{Aut}(\zp)$.  We use this group
isomorphism to conjugate $\zf$ to a virtual endomorphism of
$\text{Aut}(\zp)$.  Abusing notation, we also denote the virtual
endomorphism of $\text{Aut}(\zp)$ by $\zf$.

\emph{Torsion in $G_p$.} The group $G_p$ is an $F$-group in the
terminology of Lyndon and Schupp's book \cite{LS}.  Proposition 6.2 of
\cite{LS} implies that every torsion element of $G_p$ is conjugate to
an element with a representative loop which is either peripheral or
inessential.  So every torsion element of $\text{Aut}(\zp)$ fixes a
branch point of $\zp$.

\emph{The lift $F$ of $f^{-1}$.} The multifunction $f^{-1}$ lifts to
a genuine function $F\co D\to D$ via $\zp$.  Recall that $q$ is a
point in $S^2$ such that $f(q)=p$ and that $\zb$ is a path in
$S^2\setminus P_f$ from $p$ to $q$.  Let $\widetilde{\zb}$ be the lift
of $\zb$ to $D$ based at $\widetilde{p}$, and let $\widetilde{q}$ be
the terminal endpoint of $\widetilde{\zb}$.  We may, and do, choose
$F$ so that $F(\widetilde{p})=\widetilde{q}$.

We conclude this section with the following theorem which relates
$\zf$ and $F$.

\begin{thm}\label{thm:fnleqn} If $\zs$ is an element of the domain of
the virtual endomorphism $\zf\co \text{Aut}(\zp)\dashrightarrow
\text{Aut}(\zp)$, then $\zf(\zs)\circ F=F\circ \zs$.
\end{thm}
  \begin{proof} Let $\zs$ be in the domain of $\zf$.  We recall the
definition of $\zf(\zs)$.  Let $\zg$ be a loop in $S^2\setminus P_f$
based at $p$ which represents the element $g\in G_p$ which maps to
$\zs$ under the isomorphism from $G_p$ to $\text{Aut}(\zp)$.  Let
$\widetilde{\zg}$ be the lift of $\zg$ based at $\widetilde{p}$.  See
Figure~\ref{fig:fnleqn}.  Then the terminal endpoint of
$\widetilde{\zg}$ is $\zs(\widetilde{p})$.  We have a path $\zb$ in
$S^2\setminus P_f$ based at $p$ and ending at $q$.  We have the lift
$\widetilde{\zb}$ of $\zb$ based at $\widetilde{p}$ and ending at
$\widetilde{q}=F(\widetilde{p})$.  We have that $\zf(g)$ is the
homotopy class of $\zb f^{-1}(\zg)[q]\zb^{-1}$.  So
$\zf(\zs)(\widetilde{p})$ is the terminal endpoint of the path which
traverses $\widetilde{\zb}$, then traverses $F(\widetilde{\zg})$ and
then traverses the inverse of the lift $\widetilde{\zb}'$ of $\zb$
which ends at $F(\zs(\widetilde{p}))$.  Since $\zf(\zs)\circ
\widetilde{\zb}$ is a lift of $\zb$ based at
$\zf(\zs)(\widetilde{p})$, it follows that $\zf(\zs)\circ
\widetilde{\zb}=\widetilde{\zb}'$.  Comparing terminal endpoints of
these two equal paths, we find that
$\zf(\zs)(F(\widetilde{p}))=F(\zs(\widetilde{p}))$.  So $\zf(\zs)\circ
F$ and $F\circ\zs$ both lift $f^{-1}$, and they agree at one point,
$\widetilde{p}$.  Thus they are equal.

\begin{figure}
\centering
\includegraphics{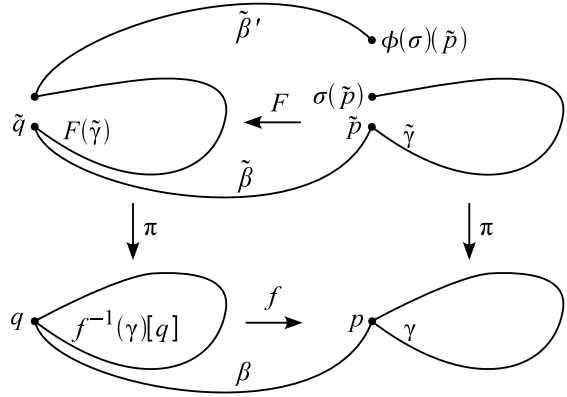}
\caption{Proving Theorems ~\ref{thm:fnleqn} and \ref{thm:endo}}
\label{fig:fnleqn}
\end{figure}

This proves Theorem~\ref{thm:fnleqn}.

\end{proof}

\section{Contracting finite subdivision rules}
\label{sec:defns}\nosubsections

In this section we assume that $f\co S^2\to S^2$ is a Thurston map
which is the subdivision map of a finite subdivision rule $\cR$.  We
fix more definitions and notation leading to the notion of contracting
finite subdivision rule.

\emph{The finite subdivision rule $\cR$.} The 2-sphere $S^2$ has the
structure of a CW complex $\cR^0(S^2)$.  Every postcritical point of
$f$ is a vertex of $\cR^0(S^2)$, and $f$ maps the 1-skeleton of
$\cR^0(S^2)$ into itself.  This leads to an infinite sequence
$\cR^0(S^2)$, $\cR^1(S^2)$, $\cR^2(S^2),\ldots$ of cell structures,
each one refining the previous one, such that $f$ maps every open cell
of $\cR^n(S^2)$ homeomorphically onto an open cell of $\cR^{n-1}(S^2)$
for every positive integer $n$.  We refer to the cells of $\cR^n(S^2)$
as cells of \emph{level} $n$.

\emph{The cell structure $\cR^{-1}(S^2)$.} We choose a tree in the
1-skeleton of $\cR^0(S^2)$ which contains $P_f$.  For neatness, we
require that every leaf of this tree is in $P_f$.  The vertices and
edges of this tree together with the 2-cell which they determine
provide $S^2$ with another cell structure, which we denote by
$\cR^{-1}(S^2)$.  When dealing with $\cR^{-1}(S^2)$, we abuse
terminology by referring to $S^2$ as a tile $t$ (with edge
identifications).  By the interior of $t$, we mean the points of $S^2$
which are not in the 1-skeleton of $\cR^{-1}(S^2)$.  This is an open
topological disk.

\emph{Lifting the cell structures to $D^*$.} We would like to use
$\zp$ to lift our cell structures on $S^2$ to cell structures on the
orbifold universal covering space $D$.  This is not always possible
because in general $\zp$ does not surject onto $S^2$.  It surjects
onto $S^2\setminus P_f^\infty$.  Consequently, $D$ is generally in a
certain sense missing some vertices.  So we enlarge $D$ by adjoining a
set $C$ (cusps) to it in such a way that $\zp$ extends naturally to
$C$, mapping $C$ to $P_f^\infty$.  We continue to use $\zp$ to denote
this extension.  Let $D^*=D\cup C$.  The cell structures
$\cR^{-1}(S^2), \cR^0(S^2),\ldots$ on $S^2$ now lift via $\zp$ to cell
structures $\cR^{-1}(D^*), \cR^0(D^*),\ldots$ on $D^*$.  (This
determines the topology of $D^*$.)  The elements of $\text{Aut}(\zp)$
extend to $D^*$ and preserve these cell structures.

\emph{The generating set $S$.} Let $t$ be the tile of $\cR^{-1}(S^2)$.
In this paragraph we require that $p$ is not in the 1-skeleton of
$\cR^{-1}(S^2)$, that is, $p$ is in the interior of $t$.  Let
$\widetilde{p}\in D$ be a lift of $p$.  Let $\widetilde{t}$ be the
lift of $t$ containing $\widetilde{p}$.  The neatness assumption in
the definition of $\cR^{-1}(S^2)$ implies that every edge of
$\widetilde{t}$ is on the boundary of $\widetilde{t}$.  These edges
can be partitioned into pairs $\{e,e'\}$ of distinct edges such that
$\zs(e)=e'$ for some $\zs\in \text{Aut}(\zp)$.  The set
$\widetilde{S}$ of all such elements $\zs$, one for every such pair of
edges, forms a generating set for $\text{Aut}(\zp)$.  Using the
obvious bijection between the $\text{Aut}(\zp)$-orbits of
$\widetilde{p}$ and $\widetilde{t}$, we obtain a corresponding subset
$S$ of $G_p$, which is a generating set for $G_p$.  We use the
notation $\left\|\cdot \right\|$ to denote lengths of elements of both
$G_p$ and $\text{Aut}(\zp)$ with respect to these generating sets.

\emph{The fat path pseudometrics.} Let $n\in \{-1,0,1,\ldots\}$.  Let
$\zg$ be a curve in $D$.  (We emphasize that $\zg$ contains no
cusps.)  We define the length $\ell _n(\zg)$ of $\zg$ relative to
$\cR^n(D^*)$ to be 1 less than the number of tiles of $\cR^n(D^*)$
which meet $\zg$.  The fat path pseudometric $d_n$ on $D$ relative
to $\cR^n(D^*)$ is defined so that $d_n(x,y)$ is the minimum such
length of a curve in $D$ joining the points $x,y\in D$.  (A fat
path in $\cR^n(D^*)$ is the union of all tiles of $\cR^n(D^*)$
which meet a given curve in $D$.)  These pseudometrics are not
metrics because every two points in the interior of one tile of
$\cR^n(D^*)$ are at $d_n$-distance 0 from each other.  In addition,
the equation $d_n(x,x)=0$ fails if $x$ is in the 1-skeleton of
$\cR^n(D^*)$.  By a $d_0$-geodesic in $S^2\setminus P_f$, we mean a
curve in $S^2\setminus P_f$ whose lifts to $D$ are $d_0$-geodesics.
Let $g\in G_p$, and let $\zs\in \text{Aut}(\zp)$ be the corresponding
element given by our bijection.  If $p$ is not in the 1-skeleton of
$\cR^{-1}(S^2)$, then $d_{-1}(\widetilde{p},\zs(\widetilde{p}))=
\left\|\zs\right\|=\left\|g\right\|$.

\emph{Decompositions of curves.} Let $n$ be a nonnegative integer.
Let $\zg$ be a curve in $S^2\setminus P_f$ beginning at the point $x$
and ending at the point $y$.  We wish to speak of the level $n$
decomposition of $\zg$.  For this we always, at least implicitly,
assume that $\zg$ meets the 1-skeleton of $\cR^n(S^2)$ in finitely
many, say $m$, points, and these points are not vertices of
$\cR^n(S^2)$.  So $\zg$ decomposes as the concatenation of curves
$\zg_0,\dotsc,\zg_m$, where $\zg_0$ begins at $x$ and $\zg_m$ ends at
$y$.  If $x$ lies in a level $n$ edge, then $\zg_0$ is constant.
Likewise if $y$ lies in a level $n$ edge, then $\zg_m$ is constant.
We call $\zg_0,\dotsc,\zg_m$ the level $n$ \emph{segments} of $\zg$.
Every segment $\zg_0,\dotsc,\zg_m$ is contained in a level $n$ tile.
This concatenation $\zg_0\cdots \zg_m$ is the \emph{level $n$
decomposition} of $\zg$.  The segments $\zg_0$ and $\zg_m$ are the
\emph{outer segments}, and $\zg_1,\dotsc,\zg_{m-1}$ are the
\emph{inner segments}.  We say that $\zg$ is \emph{taut} (for level
$n$) if it intersects the 1-skeleton of $\cR^n(S^2)$ transversely and
if the following condition is satisfied.  For every inner segment
$\za$ of $\zg$ there exists a tile type $s$ to which $\za$ pulls back.
We require that the endpoints of this pullback of $\za$ lie in
different edges of $s$.  We define level $n$ decomposition and
tautness of curves in $D$ analogously.

These notions are preserved by $\zp\co D\to S^2$.  If $\zg$ is a taut
$d_0$-geodesic arc in $S^2\setminus P_f$ with a level 0 decomposition
having $m+1$ segments, then $m$ is the $d_0$-length of every lift of
$\zg$ to $D$.  We are especially interested in the case in which the
level 0 and level $n$ decompositions of $\zg$ are equal for some
positive integer $n$.  This is illustrated in Figure~\ref{fig:equal},
where level 0 edges are drawn with solid line segments and level $n=1$
edges not contained in these are drawn with dashed line segments.  The
curve $\zg$ in Figure~\ref{fig:equal} is taut.

  \begin{figure} 
\centerline{\includegraphics{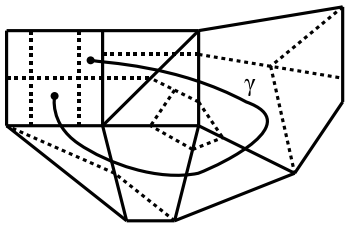}} 
\caption{A curve $\zg$ whose level 0 and level 1 decompositions are
equal}
\label{fig:equal}
  \end{figure}

\emph{Contracting finite subdivision rules.} We say that the finite
subdivision rule $\cR$ is contracting if there exist positive integers
$M$ and $n$ such that the following condition is satisfied.  If $\zg$
is a taut $d_0$-geodesic in $S^2\setminus P_f$ with level 0 and level
$n$ decompositions which are equal, then the number of level 0
segments in $\zg$ is at most $M$.\\

\noindent{\bf Remarks:}
\begin{enumerate}
\item It is immediately clear from the definition that $\cR$ is
contracting if and only if the following more combinatorial-flavored
condition holds.  Note first that a taut curve corresponds to an
edge-path without backtracking.  Next, an edge-path in the dual graph
to the tiling of $D$ at level $n$ determines in a canonical fashion a
corresponding edge-path at level $n-1$. Their lengths are the same if
and only if no two consecutive level $n$ tiles are contained in the
same parent tile. Finally, here is the equivalent condition: {\em
there exist positive integers $M, n$ such that for each level $n$
geodesic edge-path of length $M$ in the dual tiling in $D$, the
corresponding edge-path at level $0$ has length strictly less than
$M$. }

\item In light of the previous remark, it is clear that contraction is
a combinatorial condition.

\item Unlike property (M0comb) and the other properties given in \S 3,
contraction is not a condition defined solely in terms of local data.
Indeed, we are unaware of an algorithm which decides when an fsr is
contracting; see \S 9.

\item We chose the original definition of contraction to make it clear
how to attempt to check the condition: one starts enumerating curves
of length $M=1, 2, 3, \ldots$ at level $0$ and examining their
decompositions at increasing levels.
\end{enumerate}

The rest of this section is devoted to discussing contracting finite
subdivision rules.  Our interest in contracting finite subdivision
rules lies in the fact that Theorem~\ref{thm:cntrn} below shows that
$\cR$ is contracting if and only if $\zf$ is contracting.  To gain a
feeling for contracting finite subdivision rules, we next prove that
if $\cR$ is edge and vertex separating, then $\cR$ is contracting.

\begin{prop}\label{prop:cmblexpn} Let $\cR$ be a finite subdivision
rule which is edge separating and vertex separating.  Then $\cR$ is
contracting. 
\end{prop}
  \begin{proof} Because $\cR$ is edge separating, there exists a
positive integer $k$ such that for every tile type $t$ of $\cR$ and
disjoint edges $e_1$ and $e_2$ of $t$, no tile of $\cR^k(t)$ contains
subedges of $e_1$ and $e_2$.  Because $\cR$ is vertex separating,
there exists a positive integer $m$ such that the open
$\cR^m(S^2)$-stars of the vertices of $\cR^0(S^2)$ are disjoint.  Let
$n$ be an integer such that $n\ge k+m$ and if $x$ is a vertex of
$\cR^0(S^2)$ with $\zn(x)=\infty$, then the valence of $x$ in
$\cR^n(S^2)$ is greater than the valence of $x$ in $\cR^0(S^2)$.  Let
$M$ be a positive integer which is greater than the valence of every
vertex of $\cR^0(S^2)$ and every vertex of $\cR^0(D^*)$ in $D$.

Now let $\zg$ be a taut $d_0$-geodesic in $S^2\setminus P_f$ with level 0
and level $n$ decompositions which are equal.  We will show that $\zg$
has at most $M$ level 0 segments.

In this paragraph, we show that $\zg$ is contained in the
$\cR^m(S^2)$-open star of some level 0 vertex.  Let $\za$ be an inner
level $n$ segment of $\zg$.  The endpoints of $\za$ lie in interiors
of level $n$ edges.  Because $n>m$, the endpoints of $\za$ also lie in
level $m$ edges $e_1$ and $e_2$.  Because $n\ge k+m$, the choice of
$k$ implies that $e_1$ and $e_2$ are consecutive edges of some level
$m$ tile and hence are not disjoint.  Because $e_1$ and $e_2$ are
contained in level 0 edges, $e_1\cap e_2$ is contained in a level 0
edge.  Because $\zg$ is taut, it follows that $e_1\cap e_2$ contains a
level 0 vertex $x$.  Hence $\za$ is contained in the open
$\cR^m(S^2)$-star of $x$.  Since such open stars are disjoint, it
follows that $\zg$ is contained in the open $\cR^m(S^2)$-star of $x$.
Moreover, because $\zg$ is taut, it winds around $x$ in the sense that
its inner segments join consecutive level 0 edges which contain $x$
without backtracking.

If $\zn(x)<\infty$, then because $\zg$ is a $d_0$-geodesic, a lift of
$\zg$ to $D$ cannot wind completely around the corresponding lift of
$x$.  Hence if $\zn(x)<\infty$, then $\zg$ has at most $M$ segments.
If $\zn(x)=\infty$, then $\zg$ cannot wind completely around $x$
because its level $n$ decomposition equals its level 0 decomposition
but the valence of $x$ in $\cR^n(S^2)$ is greater than the valence of
$x$ in $\cR^0(S^2)$.  So again, $\zg$ has at most $M$ segments.

This proves Proposition~\ref{prop:cmblexpn}.

\end{proof}

\begin{cor}\label{cor:cmblexpn} Every combinatorially expanding finite
subdivision rule is contracting.
\end{cor}

We will see in Theorem~\ref{thm:ratcntrg} that if $\cR$ is a finite
subdivision rule whose subdivision map is a Thurston map which is
Thurston equivalent to a rational map, then $\cR$ is contracting.  So
the finite subdivision rules of Examples~\ref{ex:ex1} and \ref{ex:tr6}
are both contracting because their subdivision maps are Thurston
equivalent to rational maps.  Neither of these finite subdivision
rules is combinatorially expanding.  The property of being
combinatorially expanding is much stronger than the property of being
contracting.

\section{The contraction theorems}\label{sec:twothms}\nosubsections

We prove two theorems in this section.  The first one gives a strong
form of the assertion that the contraction ratio of the orbifold
fundamental group virtual endomorphism $\zf$ is at most 1 for every
Thurston map which is the subdivision map of a finite subdivision rule
$\cR$.  The second theorem shows that $\zf$ is contracting if and only
if $\cR$ is contracting.

We maintain the notation and terminology from Sections~\ref{sec:endo}
and \ref{sec:defns}.  The following lemma prepares for the theorems.
Statement 1 gives a fundamental inequality.  It states that the map
$F$ is distance nonincreasing with respect to the pseudometric $d_0$.
The other statements provide information about the case in which
equality is attained.

\begin{lemma}\label{lemma:funineqt} Let $x,y\in D$, and let $n$ be a
positive integer.  Then we have the following.
\begin{enumerate}
  \item $d_0(F^n(x),F^n(y))\le d_0(x,y)$
  \item If $d_0(F^n(x),F^n(y))=d_0(x,y)$, then there exist points
$x'\ne x$ and $y'\ne y$ in the open stars of $x$ and $y$ in
$\cR^0(D^*)$, respectively, such that
$d_0(F^n(x'),F^n(y'))=d_0(x',y')$.
  \item If $x$ and $y$ are not vertices of $\cR^0(D^*)$ and
$d_0(F^n(x),F^n(y)) = d_0(x,y)$, then there exists a taut
$d_0$-geodesic in $D$ joining $F^n(x)$ and $F^n(y)$ whose level 0 and
level $n$ decompositions are equal.
  \item If $x$ and $y$ are not vertices of $\cR^0(D^*)$ and they are
the endpoints of a lift via $F^n$ of a taut $d_0$-geodesic joining
$F^n(x)$ and $F^n(y)$ whose level 0 and level $n$ decompositions are
equal, then $d_0(F^n(x),F^n(y))=d_0(x,y)$.
\end{enumerate}
\end{lemma}
  \begin{proof} The function $F\co D^*\to D^*$ which lifts $f^{-1}$ is
a cellular map from $\cR^0(D^*)$ to $\cR^1(D^*)$.  Similarly, $F^n$ is
a cellular map from $\cR^0(D^*)$ to $\cR^n(D^*)$.

Let $\zg$ be a $d_0$ geodesic arc joining $x$ and $y$.  Let
$t_0,\dotsc,t_m$ be the tiles of $\cR^0(D^*)$ which meet $\zg$.  So
$m=d_0(x,y)$.  We assume that $\zg$ contains no vertex of $\cR^0(D^*)$
other than possibly $x$ and $y$, that $\zg$ meets each of
$t_0,\dotsc,t_m$ in an arc and that it intersects the 1-skeleton of
$\cR^0(D^*)$ transversely.  So $\zg$ is taut.  We maintain these
assumptions for the entire proof of Lemma~\ref{lemma:funineqt}.

Then $F^n(t_0),\dotsc,F^n(t_m)$ are the tiles of $\cR^n(D^*)$ which
meet $F^n(\zg)$.  Each of $F^n(t_0),\dotsc,F^n(t_m)$ is contained in a
unique tile of $\cR^0(D^*)$.  This proves statement 1.

To prove statement 2, suppose that $d_0(F^n(x),F^n(y))=d_0(x,y)$. Let
$d=d_0(x,y)$.  Then $d$ tiles of $\cR^0(D^*)$ meet $\zg$.  So at most
$d$ tiles of $\cR^n(D^*)$ meet $F^n(\zg)$.  So at most $d$ tiles of
$\cR^0(D^*)$ meet $F^n(\zg)$.  Now we use the assumption that
$d_0(F^n(x),F^n(y))=d_0(x,y)$.  This assumption implies that these
inequalities are equalities.  So the map from the set of tiles of
$\cR^0(D^*)$ which contain $\zg$ to the set of tiles of $\cR^0(D^*)$
which contain $\cR^0(F^n(\zg))$ is a bijection.  Hence if $x'$ and
$y'$ are any points of $\zg$, then $d_0(F^n(x'),F^n(y'))=d_0(x',y')$.
This proves statement 2.

In the situation of statement 3, the curve $F^n(\zg)$ is a taut
$d_0$-geodesic joining $F^n(x)$ and $F^n(y)$ whose level 0 and level
$n$ decompositions are equal.  This proves statement 3.

In the situation of statement 4, we have that
$d_0(F^n(x),F^n(y))=\\d_n(F^n(x),F^n(y))$ because $F^n(x)$ and $F^n(y)$
are joined by a taut $d_0$-geodesic whose level 0 and level $n$
decompositions are equal.  Because $x$ and $y$ are joined by a lift of
such a $d_0$-geodesic, $d_0(x,y)\le
d_n(F^n(x),F^n(y))=d_0(F^n(x),F^n(y))$.  This inequality combined with
statement 1 yields statement 4.

This proves Lemma~\ref{lemma:funineqt}.

\end{proof}

The next result shows that the virtual endomorphism on the orbifold
fundamental group associated to a subdivision map of an fsr cannot
expand word lengths very much.

\begin{thm}\label{thm:endo} Let $f\co S^2\to S^2$ be a Thurston map
which is Thurston equivalent to the subdivision map of a finite
subdivision rule.  Then there exist
\begin{enumerate}
  \item nonnegative integers $a$ and $b$;
  \item a choice of orbifold fundamental group virtual endomorphism
$\zf\co G_p\dashrightarrow G_p$ associated to $f$;
  \item a finite generating set for the orbifold fundamental group
$G_p$ with associated length function $\left\|\cdot \right\|$
\end{enumerate}
such that
  \begin{equation*}
\left\|\zf^n(g)\right\|\le a \left\|g\right\|+bn
  \end{equation*}
for every nonnegative integer $n$ and $g\in \text{dom}(\zf^n)$.  If
the subdivision map takes some tile of level 1 to the tile of level 0
which contains it, then we may take $b=0$.  If there is only one tile
of level 0, then we may take $a=1$.
\end{thm}

In a sequel to this work \cite{fpp:exppropii}, we will use Theorem
\ref{thm:endo} to show that Thurston maps with certain properties are
not combinatorially equivalent to subdivision maps of subdivision
rules.  For example, if $f$ has a piece of its canonical decomposition
on which the associated restriction is a pseudo-Anosov homeomorphism
\cite{kmp:cds}, then $f$ cannot be equivalent to the subdivision map
of an fsr.

  \begin{proof} The theorem is insensitive to Thurston equivalence, so
we assume that $f$ is the subdivision map of a finite subdivision rule
$\cR$.  

We return to the situation of Section~\ref{sec:endo} and the beginning
of Section~\ref{sec:defns}.  We choose the basepoint $p$ to be any
point in the interior of a level 0 tile.  The point $q$ is a point of
$S^2$ such that $f(q)=p$, and $t$ is the tile of $\cR^{-1}(S^2)$.
Both $p$ and $q$ are in the interior of $t$.  We choose a path $\zb$
from $p$ to $q$ also in the interior of $t$.  These choices of $p$,
$q$ and $\zb$ determine $\zf$.

Now let $\zg$ be a loop at $p$ representing $g\in \text{dom}(\zf)$.
See Figure~\ref{fig:fnleqn}.  Then $\zf(g)$ is the homotopy class of
$\zb f^{-1}(\zg)[q]\zb^{-1}$.  Let $\widetilde{\zg}$ be the lift of
$\zg$ to $D$ based at the lift $\widetilde{p}$ of $p$, and let $\zs\in
\text{Aut}(\zp)$ be the element corresponding to $g$.  Then the
endpoints of $\widetilde{\zg}$ are $\widetilde{p}$ and
$\zs(\widetilde{p})$.  Let $\widetilde{\zb}$ be the lift to $D$ of
$\zb$ based at $\widetilde{p}$, and let $\widetilde{t}$ be the lift of
$t$ which contains $\widetilde{p}$.  Then $\left\|\zf(g)\right\|$
equals the $d_{-1}$-distance between the endpoints of the following
path.  Traverse the path $\widetilde{\zb}$ in $\widetilde{t}$, follow
this by $F(\widetilde{\zg})$ and follow this by the lift of $\zb^{-1}$
based at $F(\zs(\widetilde{p}))$ in some open tile of $\cR^{-1}(D^*)$.

We postpone consideration of this $d_{-1}$-distance to consider the
corresponding $d_0$-distance.  This $d_0$-distance is at most
$d_0(F(\widetilde{p}),F(\zs(\widetilde{p})))+b$, where $b$ is twice
the $d_0$-length of $\widetilde{\zb}$.  Statement 1 of
Lemma~\ref{lemma:funineqt} with $n=1$ implies that this $d_0$-distance
is at most $d_0(\widetilde{p},\zs(\widetilde{p}))+b$.  So we begin
with $\widetilde{p}$ and a point at distance
$d_0(\widetilde{p},\zs(\widetilde{p}))$ from it and we end with
$\widetilde{p}$ and a point at distance at most
$d_0(\widetilde{p},\zs(\widetilde{p}))+b$ from it.  Now we iterate.
Replacing $\zf$ by $\zf^n$ for some nonnegative integer $n$, we end
with $\widetilde{p}$ and a point at distance at most
$d_0(\widetilde{p},\zs(\widetilde{p}))+bn$ from it.  If $a$ is the
number of tiles in $\cR^0(S^2)$, then
$d_0(\widetilde{p},\zs(\widetilde{p}))\le
ad_{-1}(\widetilde{p},\zs(\widetilde{p}))=a\left\|g\right\|$.  Thus if
$g\in \text{dom}(\zf^n)$ for some nonnegative integer $n$, then
  \begin{equation*}
\left\|\zf^n(g)\right\|\le a\left\|g\right\|+bn.
  \end{equation*}

This proves the first assertion of the theorem.  Suppose that $f$ maps
some tile of $\cR^1(S^2)$ to the tile of $\cR^0(S^2)$ which contains
it.  Then we may choose $p$, $q$ and $\zb$ to be in this tile of
$\cR^0(S^2)$.  In this case $b=0$.  The rest of the theorem is now
clear.

\end{proof}

\begin{cor}\label{cor:endo} Let $f\co S^2\to S^2$ be a Thurston map
which is Thurston equivalent to the subdivision map of a finite
subdivision rule.  Then the contraction ratio of the associated
orbifold fundamental group virtual endomorphism is at most 1.
\end{cor}

In general, the term $bn$ in Theorem~\ref{thm:endo} tends to $\infty$
as $n\to \infty$.  This is undesireable behavior.  However, if $f$
maps some tile of level 1 to the tile of level 0 which contains it,
then we may take $b=0$.  The next result shows that this can be
achieved by passing to an iterate of $f$.

\begin{prop}\label{prop:bzero} Let $f\co S^2\to S^2$ be a Thurston map
which is the subdivision map of a finite subdivision rule $\cR$.  Then
there exists a positive integer $n$ such that $f^n$ maps some tile of
level $n$ to the tile of level 0 which contains it.
\end{prop}
  \begin{proof} We construct a directed graph $\zG$.  The vertices of
$\zG$ are the tiles of level 0.  There exists a directed edge from
tile $t_1$ to tile $t_2$ if and only if $t_1$ contains a tile $s$ of
level 1 such that $f(s)=t_2$.  The graph $\zG$ has finitely many
vertices and at least one directed edge emanating from every vertex.
Hence it has a cycle, say of length $n$.  If $t$ is a tile in this
cycle, then $t$ contains a tile $s$ of level $n$ such that
$f^n(s)=t$.  This proves Proposition~\ref{prop:bzero}.

\end{proof}

Here is our second contraction theorem.

\begin{thm}\label{thm:cntrn} Let $f\co S^2\to S^2$ be a Thurston map
which is Thurston equivalent to the subdivision map of a finite
subdivision rule $\cR$.  Let $\zr_\zf$ be the contraction ratio of the
associated orbifold fundamental group virtual endomorphism $\zf$.
Then
\begin{enumerate}
  \item $\zr_\zf=1$ if $\cR$ is not contracting;
  \item $\zr_\zf<1$ if $\cR$ is contracting.
\end{enumerate} 
Thus $\zf$ is contracting if and only if $\cR$ is contracting.
\end{thm}
  \begin{proof} The last sentence follows from statements 1 and 2
together with Proposition 2.11.11 from Nekrashevych's book \cite{N},
using the fact that $\zf$ is level transitive.

Like Theorem~\ref{thm:endo}, this theorem is insensitive to Thurston
equivalence, so we assume that $f$ is the subdivision map of a finite
subdivision rule $\cR$.  We return to the setting of
Section~\ref{sec:endo} and the beginning of Section~\ref{sec:defns}. 

We next prove statement 1.  Corollary~\ref{cor:endo} shows that
$\zr_\zf\le 1$.  So it suffices to prove that $\zr_\zf\ge 1$.  

Suppose that $\cR$ is not contracting.  Let $n$ be a positive
integer.  Then for every positive integer $M$ there exists a taut
$d_0$-geodesic $\zg$ in $S^2\setminus P_f$ whose level 0 and level $n$
decompositions are equal and $\zg$ has more than $M$ segments.

In this paragraph we show that we may assume that $\zg$ is a closed
curve.  Let $\widetilde{\zg}$ be a lift of $\zg$ to $D$.  Let $\za$ be
an inner segment of $\widetilde{\zg}$.  So $\za$ joins two edges $e_1$
and $e_2$ of the tile $t$ of $\cR^n(D^*)$ which contains $\za$.  We
consider a homotopy of $\widetilde{\zg}$ which changes only $\za$ and
the segments of $\widetilde{\zg}$ immediately preceding and following
$\za$.  This homotopy moves $\za$ through a family of arcs in $t$
which join $e_1$ and $e_2$.  This homotopy moves the segments
immediately preceding and following $\za$ only near $e_1$ and $e_2$.
Projecting to $S^2$, such a homotopy takes $\zg$ to another curve in
$S^2\setminus P_f$ whose level 0 and level $n$ decompositions are
equal with the same number of segments as $\zg$.  So up to such
homotopies, $\zg$ is determined by the sequence of edges of
$\cR^n(S^2)$ which it crosses.  Since there are only finitely many
such ordered pairs of edges and $M$ is arbitrary, it follows that we
may assume that $\zg$ is a closed curve.

Since $\cR^n(S^2)$ contains only finitely many tiles, there is a tile
$t$ of $\cR^n(S^2)$ for which there are infinitely many values of $M$
for which the corresponding curve $\zg$ meets $t$.  Similarly, since
$f^n(t)$ is a level 0 tile, there are infinitely many values of $n$
such that $f^n(t)$ is a fixed level 0 tile $s$.  Let $p$ be a point in
the interior of $s$.

Now let $n$ be one of the infinitely many positive integers such that
$p\in f^n(t)$.  Choose $q\in t$ such that $f^n(q)=p$.  Because $p$ is
in the interior of $s$, it follows that $q$ is in the interior of $t$.
Since there are infinitely many choices for $M$, there exist
infinitely many elements $g\in G_q$ with representative taut
$d_0$-geodesic loops whose level $n$ and level 0 decompositions are
equal.  Now choose $g\in G_q$ with a representative taut
$d_0$-geodesic loop $\zg$ whose level $n$ and level 0 decompositions
are equal.  Let $\zh=f^n(\zg)$, a loop based at $p$.  In other words,
$\zh$ is a loop in $S^2\setminus P_f$ based at $p$ and $\zg$ lifts
$\zh$ via $f^n$.  See Figure~\ref{fig:cntrn}.  Let $\widetilde{\zg}$
be a lift of $\zg$ to $D$ via $\zp$, and let $\widetilde{\zh}$ be a
lift of $\widetilde{\zg}$ to $D$ via $F^n$.  Then $\widetilde{\zh}$ is
a lift of $\zh$ to $D$ via $\zp$.  Let $\widetilde{q}$ and
$\widetilde{q}\,'$ be the initial and terminal endpoints of
$\widetilde{\zg}$.  Let $\widetilde{p}$ and $\widetilde{p}\,'$ be the
initial and terminal endpoints of $\widetilde{\zh}$.

  \begin{figure}
\centerline{\includegraphics{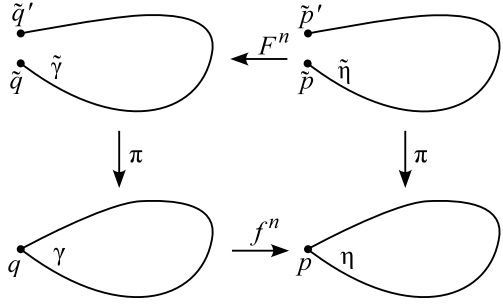}}\caption{Proving statement 1
of Theorem~\ref{thm:cntrn}}
\label{fig:cntrn}
  \end{figure}

Let $h\in G_p$ be the element represented by $\zh$.  We recall the
discussion of the iterates of $\zf$ in Section~\ref{sec:endo}.  We
begin with a path $\zb$ from $p$ to a point in $f^{-1}(p)$.  To
compute $\zf^n$, we take $n-1$ lifts of $\zb$, obtaining a point in
$f^{-n}(p)$.  If that point is $q$, then $h\in \text{dom}(\zf^n)$.  In
general, that point is not $q$.  Nonetheless, there exists a virtual
endomorphism $\zf_n\co G_p\dashrightarrow G_p$ with $h\in
\text{dom}(\zf_n)$ such that $\zf_n$ equals $\zf^n$ precomposed and
postcomposed with conjugations.  These conjugations and $\zf_n$ depend
only on $n$, not on $h$.  So the estimates which we obtain for $\zf_n$
will provide the needed estimates for $\zf^n$.

We seek a lower bound for $\left\|h\right\|$.  Let $a$ be the number
of tiles in $\cR^0(S^2)$.  Then
  \begin{equation*}
\left\|h\right\|=d_{-1}(\widetilde{p},\widetilde{p}\,')\ge 
\frac{1}{a}d_0(\widetilde{p},\widetilde{p}\,')\ge
\frac{1}{a}d_0(\widetilde{q},\widetilde{q}\,')\ge
\frac{1}{a}d_{-1}(\widetilde{q},\widetilde{q}\,')=\frac{1}{a}\left\|g\right\|,
  \end{equation*}
where the equations use the fact that $p$ and $q$ are not in the
1-skeleton of $\cR^{-1}(S^2)$ and the second inequality comes from
statement 1 of Lemma~\ref{lemma:funineqt}.  From this we conclude that
$\left\|h\right\|\to \infty$ as $\left\|g\right\|\to \infty$.

The previous paragraph provides a lower bound for $\left\|h\right\|$.
Now we obtain an upper bound as follows.
  \begin{equation*}
\begin{aligned}
\left\|h\right\| & =d_{-1}(\widetilde{p},\widetilde{p}\,')
&& \text{definitions of }\left\|\cdot \right\|\text{ and }d_{-1}\\
 & \le d_0(\widetilde{p},\widetilde{p}\,')
&& \cR^0(D^*)\text{ refines }\cR^{-1}(D^*)\\
 & = d_0(\widetilde{q},\widetilde{q}\,')
&& \text{statement 4 of Lemma~\ref{lemma:funineqt}}\\
 & \le ad_{-1}(\widetilde{q},\widetilde{q}\,')
&& \text{every level }-1\text{ tile contains }a\text{ level 0-tiles}\\
 & =a\left\|g\right\|
&& \text{definitions of }\left\|\cdot \right\|\text{ and }d_{-1}
\end{aligned}
  \end{equation*}
Moreover, $\left\|\zf_n(h)\right\|$ differs from $\left\|g\right\|$ by
at most an additive constant $K$ (for a fixed $n$) which arises from
the implicit path from $p$ to $q$.  Since there are infinitely many
choices for $g$, the lengths $\left\|g\right\|$ are unbounded.  Hence
the previous paragraph implies that the lengths $\left\|h\right\|$ are
unbounded.  So when we consider all elements $h\in \text{dom}(\zf_n)$,
we find that
  \begin{equation*}
\limsup_{\left\|h\right\|\to
\infty}\frac{\left\|\zf_n(h)\right\|}{\left\|h\right\|}\ge
\limsup_{\left\|g\right\|\to
\infty}\frac{\left\|g\right\|-K}{a\left\|g\right\|}= 
\frac{1}{a}.
  \end{equation*}
Hence
  \begin{equation*}
\limsup_{\left\|k\right\|\to
\infty}\frac{\left\|\zf^n(k)\right\|}{\left\|k\right\|}\ge
\frac{1}{a}.
  \end{equation*}
This is true for infinitely many positive integers $n$.  So taking the
limit supremum of $n$th roots, we find that $\zr_\zf\ge 1$.

This proves statement 1.

To prove statement 2, we assume that $\cR$ is contracting.  Statement
1 of Lemma~\ref{lemma:funineqt} implies that $F$ is distance
nonincreasing.  The following lemma shows that if $\cR$ is
contracting, then some iterate of $F$ is uniformly strictly distance
decreasing in the large.

\begin{lemma}\label{lemma:gmcineqt} If $\cR$ is contracting, then
there exist positive integers $n$ and $N$ and a real number $c$ with
$0<c<1$ such that
  \begin{equation*}
d_0(F^n(x),F^n(y))\le cd_0(x,y)
  \end{equation*}
for every $x,y\in D$ with $d_0(x,y)\ge N$.
\end{lemma}
  \begin{proof} Here is the only place in the proof of statement 2
that we use the fact that $\cR$ is contracting.  Because $\cR$ is
contracting, there exist positive integers $M$ and $n$ such that if
$\zg$ is a taut $d_0$-geodesic in $S^2\setminus P_f$ whose level 0 and
level $n$ decompositions are equal, then $\zg$ has at most $M$
segments.  Let $K$ be the largest valence of a vertex of $\cR^0(D^*)$
in $D$.  Let $N=M+2K$.  This determines the integers $n$ and $N$ in
the statement of the lemma.

Now let $x,y\in D$, and suppose that $d_0(F^n(x),F^n(y))\ge
d_0(x,y)$.  We will prove that $d_0(x,y)<N$.  

We first reduce to the case in which neither $x$ nor $y$ is a vertex
of $\cR^0(D^*)$.  Statement 1 of Lemma~\ref{lemma:funineqt} implies
that $d_0(F^n(x),F^n(y))=d_0(x,y)$.  Now statement 2 of
Lemma~\ref{lemma:funineqt} implies that there exist points $x'$ and
$y'$, which are not vertices of $\cR^0(D^*)$, in the open stars of $x$
and $y$ in $\cR^0(D^*)$ such that $d_0(F^n(x'),F^n(y'))=d_0(x',y')$.
So to prove that $d_0(x,y)<N$, it suffices to prove that
$d_0(x',y')<M$.

But statement 3 of Lemma~\ref{lemma:funineqt} implies that there
exists a taut $d_0$-geodesic joining $F^n(x')$ and $F^n(y')$ whose
level 0 and level $n$ decompositions are equal.  The assumption that
$\cR$ is contracting easily implies that this geodesic has at most $M$
segments.  Hence $d_0(x',y')<M$.

We have proved that if $d_0(x,y)=N$, then $d_0(F^n(x),F^n(y))<N$.  In
general, we choose a $d_0$-geodesic joining $x$ and $y$.  We decompose
it into as many initial segments of length $N$ as possible.  We obtain
this inequality for each such segment.  Statement 1 of
Lemma~\ref{lemma:funineqt} applies to the remainder, and so if
$d_0(x,y)\ge N$, then
  \begin{equation*}
\begin{aligned}
d_0(F^n(x),F^n(y)) & \le (N-1)\left\lfloor\frac{d_0(x,y)}{N}\right\rfloor+
d_0(x,y)-N\left\lfloor\frac{d_0(x,y)}{N}\right\rfloor\\
 & \le d_0(x,y)-\left\lfloor\frac{d_0(x,y)}{N}\right\rfloor
\le d_0(x,y)-\frac{1}{2N}d_0(x,y)\\
 & \le (1-\frac{1}{2N})d_0(x,y).
\end{aligned}
  \end{equation*}

This proves Lemma~\ref{lemma:gmcineqt} with $c=1-\frac{1}{2N}$.
     
\end{proof}

With Lemma~\ref{lemma:gmcineqt} in hand, we proceed as follows.  The
definition of the contraction ratio $\zr_\zf$ involves the limit
supremum of fractions of the form
$\left\|\zf^m(g)\right\|/\left\|g\right\|$ for positive integers $m$.
Write $m=kn+l$ with $n$ as in Lemma~\ref{lemma:gmcineqt} and
nonnegative integers $k$ and $l$ with $0\le l<n$.  Then
  \begin{equation*}
\frac{\left\|\zf^m(g)\right\|}{\left\|g\right\|}=
\frac{\left\|\zf^{kn}(\zf^l(g))\right\|}{\left\|\zf^l(g)\right\|}
\frac{\left\|\zf^l(g)\right\|}{\left\|g\right\|}.
  \end{equation*}
Using the definitions, Figure~\ref{fig:fnleqn}, statement 1 of
Lemma~\ref{lemma:funineqt} and the fact that there are only finitely
many possibilities for $l$, one sees that the last of these three
fractions is bounded for $g\in \text{dom}(\zf^m)$.  It follows that to
prove that $\zr_\zf<1$, we may replace $f$ by $f^n$.  So we have, as
in Lemma~\ref{lemma:gmcineqt}, that $d_0(F(x),F(y))\le cd_0(x,y)$ if
$d_0(x,y)\ge N$.

We choose a basepoint $p$ in the interior of some tile of
$\cR^0(S^2)$.  Let $\widetilde{p}$ be a lift of $p$ to $D$.  Instead
of working with the virtual endomorphism of $G_p$, we work with the
equivalent virtual endomorphism of $\text{Aut}(\zp)$.

Let $m$ be a positive integer.  Finally, set
$K=2d_0(F^m(\widetilde{p}),\widetilde{p})$ and choose $\zs\in
\text{Aut}(\zp)$ such that $\left\|\zf^m(\zs)\right\|\ge N+K$.  Then
we have the following, keeping in mind that when
Lemma~\ref{lemma:gmcineqt} is applied, the choice of $\zs$ implies
that the right side of the inequality is at least $N+K$.

  \begin{equation*}
\begin{aligned}
\left\|\zf^m(\zs)\right\| \hspace*{5pt}
& =&&\parbox{.8\linewidth}{$d_{-1}(\zf^m(\zs)(\widetilde{p}),\widetilde{p})$
   \hfill definitions of $\left\|\cdot \right\|$ and $d_{-1}$}\\
& \le&&\parbox{.8\linewidth}{$d_0(\zf^m(\zs)(\widetilde{p}),\widetilde{p})$
  \hfill $\cR^0(D^*)$ refines $\cR^{-1}(D^*)$}\\
& \le&&\parbox{.8\linewidth}{$d_0(\zf^m(\zs)(\widetilde{p}),
  \zf^m(\zs)(F^m(\widetilde{p})))+
  d_0(\zf^m(\zs)(F^m(\widetilde{p})),F^m(\widetilde{p}))$}\\
& &&\parbox{.8\linewidth}{$
  +\, d_0(F^m(\widetilde{p}),\widetilde{p})$
  \hfill triangle inequality}\\
& =&&\parbox{.8\linewidth}{$d_0(\zf^m(\zs)(F^m(\widetilde{p})),
  F^m(\widetilde{p}))+K$ \hfill $\zf^m(\zs)$ is a $d_0$-isometry}\\
& =&&\parbox{.8\linewidth}{$d_0(F(\zf^{m-1}(\zs)(F^{m-1}(\widetilde{p}))),
  F(F^{m-1}(\widetilde{p})))+K$ \hfill Theorem~\ref{thm:fnleqn}}\\
& \le&&\parbox{.8\linewidth}{$ cd_0(\zf^{m-1}(\zs)(F^{m-1}(\widetilde{p})),
  F^{m-1}(\widetilde{p}))+K$ \hfill choice of $\zs$,}\\
& &&\parbox{.8\linewidth}{\hfill statement 1 of
  Lemma~\ref{lemma:funineqt} and Lemma~\ref{lemma:gmcineqt}}\\
& \hspace*{2mm}\vdots &&\\
&\le&&\parbox{.8\linewidth}{$c^md_0(\zs(\widetilde{p}),\widetilde{p})+K$
  \hfill etc.}\\
& \le&&\parbox{.8\linewidth}{$ac^md_{-1}(\zs(\widetilde{p}),\widetilde{p})+
  K$ \hfill every level $-1$ tile}\\
& &&\parbox{.8\linewidth}{\hfill contains $a$ level 0 tiles}\\
& =&&\parbox{.8\linewidth}{$ac^m\left\|\zs\right\|+K$ \hfill
  definitions of $\left\|\cdot \right\|$ and $d_{-1}$}\\
\end{aligned}
  \end{equation*}
Thus
  \begin{equation*}
\limsup_{\left\|\zs\right\|\to \infty}
\frac{\left\|\zf^m(\zs)\right\|}{\left\|\zs\right\|}\le 
\limsup_{\left\|\zs\right\|\to \infty}\left(ac^m+
\frac{K}{\left\|\zs\right\|}\right)
= ac^m.
  \end{equation*}
Taking the limit supremum of $m$-th roots, we conclude that $\zr_\zf<1$.

This proves Theorem~\ref{thm:cntrn}.

\end{proof}

\begin{thm}\label{thm:ratcntrg} Let $\cR$ be a finite subdivision rule
whose subdivision map $f$ is a Thurston map which is Thurston
equivalent to a map $g$ with the following property.  As for $f$, let
$P_g^\infty$ denote the set of orbifold points for $g$ with orbifold
weight infinity.  There is a neighborhood $U$ of $P_g^\infty$ and a
complete length metric on $K:=S^2-U$ such that $g^{-1}(K) \subset K$
and inverse branches of $g$ uniformly decrease lengths of curves.
Then $\cR$ is contracting.

In particular, if $f$ is equivalent to a rational map, then $\cR$ is
contracting.
\end{thm}
  \begin{proof} Theorem 6.4.4 of Nekrashevych's book \cite{N} implies
that the orbifold fundamental group virtual endomorphism of $f$ is
contracting.  Thus $\cR$ is contracting by Theorem~\ref{thm:cntrn}.

\end{proof}

\section{The fat path subdivision graph}\nosubsections
\label{sec:sdivngraph}

In this section we assume that $f\co S^2\to S^2$ is a Thurston map
which is the subdivision map of a finite subdivision rule $\cR$.  We
will define the fat path subdivision graph of $\cR$.  We will give a
condition on $\cR$ which is equivalent to the condition that this
graph is Gromov hyperbolic.  We will also describe the Gromov boundary
of this graph when it is hyperbolic.

As in Section~\ref{sec:defns}, we have cell complexes $\cR^{-1}(S^2)$,
$\cR^0(S^2)$, $\cR^1(S^2),\ldots$.  The \emph{fat path subdivision
graph} of $\cR$ is defined as follows.  It is a graph $\zG$ with a
vertex $v(t)$ for every tile $t$ of $\cR^n(S^2)$ for integers $n\ge
-1$.  We say that the vertex $v(t)$ \emph{represents} the tile $t$.
Let $t$ be a tile of $\cR^n(S^2)$ for some integer $n\ge -1$.  Then
$v(t)$ is joined by an edge of $\zG$ to $v(s)$ for every subtile $s$
of $t$ in $\cR^{n+1}(S^2)$.  These edges are said to be
\emph{vertical}.  The vertex $v(t)$ is also joined by an edge of $\zG$
to $v(s)$ for every tile $s\ne t$ of $\cR^n(S^2)$ which has an edge in
common with $t$.  These edges are said to be \emph{horizontal}.  The
\emph{skinny path subdivision graph} of $\cR$ is defined in the same
way except that if $s\ne t$ are tiles of $\cR^n(S^2)$, then $v(s)$ and
$v(t)$ are joined by an edge if and only if $s\cap t\ne \emptyset$.
Thus the only difference between these graphs is that in general the
fat path subdivision graph has fewer horizontal edges than the skinny
path subdivision graph.  These graphs are given metrics in the
straightforward way so that every edge has length 1.

For every integer $n\ge -1$, let $\zG_n$ be the subgraph of $\zG$
spanned by all vertices of the form $v(t)$, where $t$ is a tile of
$\cR^n(S^2)$.  Let $\zd_n$ be the path metric on $\zG_n$.  Because
$\zG_n$ might not be connected, $\zd_n$ might take the value $\infty$.
Following Rushton \cite{R}, we define a \emph{transition function}
$f_{m,n}\co \zG_n\to \zG_m$ for all integers $m,n\ge -1$ such that
$m\le n$.  Let $t$ be a tile of $\cR^n(S^2)$.  Then $f(v(t))=v(s)$,
where $s$ is the tile of $\cR^m(S^2)$ which contains $t$.  We extend
$f_{m,n}$ to the edges of $\zG_n$ in the straightforward way, so that
$f_{m,n}$ is a cellular map.

These graphs are called history graphs on page 100 of
\cite{downunder}, where they were briefly introduced.  We feel that
the terminology finite subdivision graph is a bit more descriptive and
it distinguishes our graphs from those studied by Rushton in \cite{R}.
Rushton's graphs depend not only on $\cR$ but also a choice of
$\cR$-subdivision complex $X$ and ideal cells (more about these
later).  For us $X$ is the 2-sphere $S^2$.  Our skinny path
subdivision graph is essentially Rushton's history graph for $X=S^2$
and no ideal cells.

We wish to apply two of Rushton's results, Theorems 5 and 6 of
\cite{R}.  The former gives necessary and sufficient conditions for
hyperbolicity of a graph.  The latter shows that the boundary is
homeomorphic to a certain quotient.  However, Rushton's results do not
apply directly to the fat path subdivision graph $\zG$ because Rushton
in effect works with skinny path subdivision graphs.  These
subdivision graphs are quasi-isometric if $\cR$ has bounded valence,
but they are not quasi-isometric if $\cR$ does not have bounded
valence.  To obviate this difficulty, we introduce a new finite
subdivision rule $\widehat{\cR}$ which is gotten from $\cR$ by
introducing ``ideal tiles'' at vertices of $\cR^0(S^2)$ whose valences
are unbounded under subdivision.  We will see that Rushton's skinny
path subdivision graph for $\widehat{\cR}$ with these ideal tiles is
quasi-isometric to $\zG$, so Rushton's results applied to
$\widehat{\cR}$ describe $\zG$.

In this paragraph we define the finite subdivision rule
$\widehat{\cR}$.  For this, let $v$ be a vertex of $\cR^0(S^2)$ whose
valences are unbounded under subdivision.  We call such vertices
\emph{ideal vertices} of $\cR^0(S^2)$.  The space $S^2\setminus \{v\}$
is homeomorphic to $S^2\setminus D$, where $D$ is a small closed
topological disk containing $v$ in its interior.  Using such a
homeomorphism, the cell structure of $\cR^0(S^2)$ together with the
boundary $\partial D$ of $D$ determine a cell structure on the closure
of $S^2\setminus D$.  The induced cell structure on $\partial D$
together with $v$ determine a cell structure on $D$ by joining every
vertex of $\partial D$ to $v$ with an edge.  We do this for every
ideal vertex $v$ of $\cR^0(S^2)$, in effect replacing $v$ by a closed
disk subdivided into $n$ sectors, where $n$ is the valence of $v$ in
$\cR^0(S^2)$.  One verifies that this determines a new finite
subdivision rule $\widehat{\cR}$.  See Figure~\ref{fig:rhat}.  We view
$\widehat{\cR}^n(S^2)$ as a subcomplex of $\cR^n(S^2)$ for every
nonnegative integer $n$.  Let $\widehat{f}\co S^2\to S^2$ be the
subdivision map of $\widehat{\cR}$.

  \begin{figure}
\centerline{\scalebox{.8}{\includegraphics{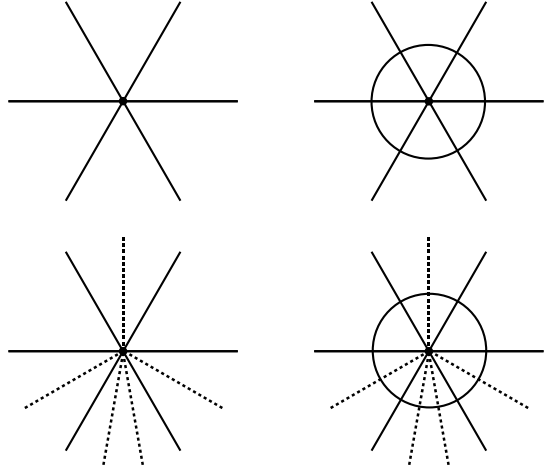}}} \caption{On the
left are $\cR^0(S^2)$ and $\cR^1(S^2)$ near $v$ and on the right are
$\widehat{\cR}^0(S^2)$ and $\widehat{\cR}^1(S^2)$ near $v$, a vertex
of $\cR^0(S^2)$ whose valences are unbounded under subdivision}
\label{fig:rhat}
  \end{figure}

Now we describe Rushton's history graph $\zG(\widehat{\cR},S^2)$.  It
has a base vertex corresponding to the vertex of $\zG_{-1}$.  It also
has one vertex for every closed cell $c$ of $\widehat{\cR}^n(S^2)$
such that $\widehat{f}^n(c)$ does not contain an ideal vertex of
$\cR^0(S^2)$.  A horizontal edge joins vertices representing 
distinct cells of $\widehat{\cR}^n(S^2)$ if and only if one cell is
contained in the other.  Containment also determines vertical edges
just as for $\zG$.  Thus there is a canonical map from the set of
vertices of $\zG$ to the set of vertices of $\zG(\widehat{\cR},S^2)$.
It is easy to see that this map extends to edges, yielding a
quasi-isometry.

Since $\zG$ and $\zG(\widehat{\cR},S^2)$ are quasi-isometric, Theorems
5 and 6 of \cite{R} applied to $\zG(\widehat{\cR},S^2)$ obtain
information about $\zG$, giving us Theorems~\ref{thm:rushtona} and
\ref{thm:rushtonb}.

\begin{thm}\label{thm:rushtona} Let $\cR$ be a finite subdivision rule
whose subdivision map is a Thurston map.  Then the fat path
subdivision graph $\zG$ of $\cR$ is Gromov hyperbolic if and only if
there exist positive integers $M$ and $n$ with the following property.
Let $u$ and $v$ be vertices of $\zG_m$ for some nonnegative integer
$m$ such that $\zd_m(u,v)\ge M$.  Recall that there is a transition
function $f_{m,m+n}\co \zG_{m+n}\to \zG_m$.  Let $u'$ and $v'$ be
vertices of $\zG_{m+n}$ such that $f_{m,m+n}(u')=u$ and
$f_{m,m+n}(v')=v$.  Then $\zd_{m+n}(u',v')>\zd_m(u,v)$.
\end{thm} 

To emphasize the similarity between the condition in
Theorem~\ref{thm:rushtona} and the condition of being contracting, we
restate the condition in Theorem~\ref{thm:rushtona} as follows.  In
our usual situation, we say that a curve $\zg$ in $S^2\setminus P_f$
is a $\zd_m$-geodesic for some nonnegative integer $m$ if the tiles of
$\cR^m(S^2)$ which meet $\zg$ are exactly the tiles represented by the
vertices of a $\zd_m$-geodesic in $\zG_m$.  The condition of
Theorem~\ref{thm:rushtona} can be restated as follows.  There exist
positive integers $M$ and $n$ with the following property.  If $\zg$
is a taut $\zd_m$-geodesic for some nonnegative integer $m$ with level
$m$ and level $m+n$ decompositions which are equal, then the number of
level $m$ segments in $\zg$ is at most $M$.  We say that $\cR$ is
\emph{graph hyperbolic} if it satisfies this condition.  The only
difference between $\cR$ being contracting and $\cR$ being graph
hyperbolic is that the former definition deals with $d_0$-geodesics
and the latter definition deals with $\zd_m$-geodesics.

We next make a definition to prepare for the next theorem.  Let $\cR$
be a finite subdivision rule whose subdivision map is a Thurston map.
Let $X$ be the topological space gotten by deleting from $S^2$ the
open star of every vertex of $\widehat{\cR}^n(S^2)$ whose valences are
unbounded under subdivision for every nonnegative integer $n$.  Now
define a relation $\sim$ on $X$ as follows.  Let $x,y\in X$.  Let
$s_n$ and $t_n$ be tiles of $\cR^n(S^2)$ which contain $x$ and $y$,
respectively, for every nonnegative integer $n$.  Let $v_n(x)$ and
$v_n(y)$ be the vertices of $\zG_n$ representing $s_n$ and $t_n$,
respectively.  Then $x\sim y$ if and only if the distances
$\zd_n(v_n(x),v_n(y))$ are bounded as $n$ varies over the nonnegative
integers.  The relation $\sim$ is an equivalence relation.

\begin{thm}\label{thm:rushtonb} Let $\cR$ be a graph hyperbolic finite 
subdivision rule whose subdivision map is a Thurston map.  Then the
boundary of $\zG$ is homeomorphic to the space $X$ defined immediately
above modulo the relation $\sim$.
\end{thm}

The following lemma gives a slightly different interpretation of the
equivalence relation $\sim$.

\begin{lemma}\label{lemma:tilde} The equivalence relation $\sim$
is generated by the relation $\approx $ defined as follows.  Let
$x,y\in X$.  Then $x\approx y$ if and only if for every nonnegative
integer $n$ there exists a tile of $\cR^n(S^2)$ which contains both
$x$ and $y$.
\end{lemma}
  \begin{proof} This is an exercise in point set topology.
\end{proof}

Let $\cR$ be a graph hyperbolic finite subdivision rule whose
subdivision map is a Thurston map.  Theorem~\ref{thm:rushtonb} and
Lemma~\ref{lemma:tilde} imply that the boundary of the fat path
subdivision graph of $\cR$ can be constructed in two steps.  In the
first step, we delete from $S^2$ an open topological disk about every
vertex of $\cR^n(S^2)$ whose valences are unbounded under subdivision
for every nonnegative integer $n$.  In the second step, we identify
two points of this space if (but not only if) some tile of
$\cR^n(S^2)$ contains both of them for every nonnegative integer $n$.
The first step is trivial if and only if $\cR$ has bounded valence.
The second step is trivial if and only if the mesh of $\cR$ approaches
0.

Using Theorem~\ref{thm:rushtona}, we next prove that if $\cR$ is
contracting, then $\cR$ is graph hyperbolic.

\begin{thm}\label{thm:cntrghpbc} Let $\cR$ be a finite subdivision
rule whose subdivision map $f$ is a Thurston map.  If $\cR$ is
contracting, then it is graph hyperbolic and hence its fat path
subdivision graph is Gromov hyperbolic.
\end{thm}
  \begin{proof} Theorem~\ref{thm:rushtona} implies that it suffices to
prove that $\cR$ is graph hyperbolic.  Let $\zG$ be the fat path
subdivision graph of $\cR$.  

Because $\cR$ is contracting, there exist positive integers $M$ and
$n$ such that the following condition is satisfied.  If $\zg$ is a
taut $d_0$-geodesic in $S^2\setminus P_f$ with level 0 and level $n$
decompositions which are equal, then the number of level 0 segments in
$\zg$ is at most $M$.

Now we begin to verify the condition of Theorem~\ref{thm:rushtona} for
these values of $M$ and $n$.  Let $u$ and $v$ be vertices of $\zG_m$
for some nonnegative integer $m$ such that $\zd_m(u,v)\ge M$.  Set
$N=\zd_m(u,v)$.  Let $u'$ and $v'$ be vertices of $\zG_{m+n}$ such
that $f_{m,m+n}(u')=u$ and $f_{m,m+n}(v')=v$.  The vertices $u$ and
$v$ represent tiles $s$ and $t$ of level $m$, and $u'$ and $v'$
represent tiles $s'$ and $t'$ of level $m+n$ with $s'\subseteq s$ and
$t'\subseteq t$.  Let $p$ and $q$ be points in the interiors of $s'$
and $t'$, respectively.  Because $\zd_m(u,v)=N$, there exists a taut
arc $\za$ in $S^2\setminus P_f$ joining $p$ and $q$ with a level $m$
decomposition containing $N+1$ segments.

Let $\widetilde{\za}$ be a curve in $D$ which lifts $\za$ via
$\zp\circ F^m\co D\to S^2$.  Let $\widetilde{p}$ and $\widetilde{q}$
be the endpoints of $\widetilde{\za}$.  Because $N+1$ tiles of
$\cR^m(S^2)$ cover $\za$, it follows that $N+1$ tiles of $\cR^0(D^*)$
cover $\widetilde{\za}$.  So $d_0(\widetilde{p},\widetilde{q})\le N$.
This inequality cannot be strict because otherwise there exists a
curve in $\cR^0(D^*)$ joining $\widetilde{p}$ and $\widetilde{q}$
covered by at most $N$ tiles of $\cR^0(D^*)$.  This curve maps by
$\zp\circ F^m$ to a curve in $\cR^m(S^2)$ joining $p$ and $q$.  It
would follow that $\zd_m(u,v)<N$, which is not true.  Hence
$d_0(\widetilde{p},\widetilde{q})=N$ and $\widetilde{\za}$ is a taut
$d_0$-geodesic.

Let $\zg=\zp(\widetilde{\za})$.  Then $\zg$ is a taut $d_0$-geodesic
in $S^2\setminus P_f$ whose level 0 decomposition has $N+1$ segments.
Because $N\ge M$, the contraction condition which $\cR$ satisfies
implies that the level $n$ decomposition of $\zg$ does not equal the
level 0 decomposition of $\zg$.  In other words, the level $n$
decomposition of $\zg$ has more than $N+1$ segments.  In turn, the
level $m+n$ decomposition of $\za$ has more than $N+1$ segments.  This
implies that $p$ and $q$ cannot be joined by a curve whose level $m+n$
decomposition has at most $N+1$ segments.  Thus
$\zd_m(u',v')>N=\zd_m(u,v)$.

This proves Theorem~\ref{thm:cntrghpbc}.

\end{proof}

See the discussion after Theorem~\ref{thm:levy} for examples which
show that it is possible for a finite subdivision rule to be graph
hyperbolic without being contracting.

\begin{ex}\label{ex:nhpbc} In practice, finite subdivision rules which
one encounters whose subdivision maps are Thurston maps are almost all
graph hyperbolic.  Here is an example of a Thurston map $f$ which is the
subdivision map of a finite subdivision rule $\cR$ which is not graph
hyperbolic.  The 1-skeleton of $\cR^0(S^2)$ is a simple closed curve
decomposed into four edges $a$, $b$, $c$, $d$.  So $\cR^0(S^2)$ has
two tiles, $t_1$ and $t_2$.  The subdivisions of $t_1$ and $t_2$,
which are reflections of each other, are shown in
Figure~\ref{fig:nhpbc}.

  \begin{figure} \centerline{\includegraphics{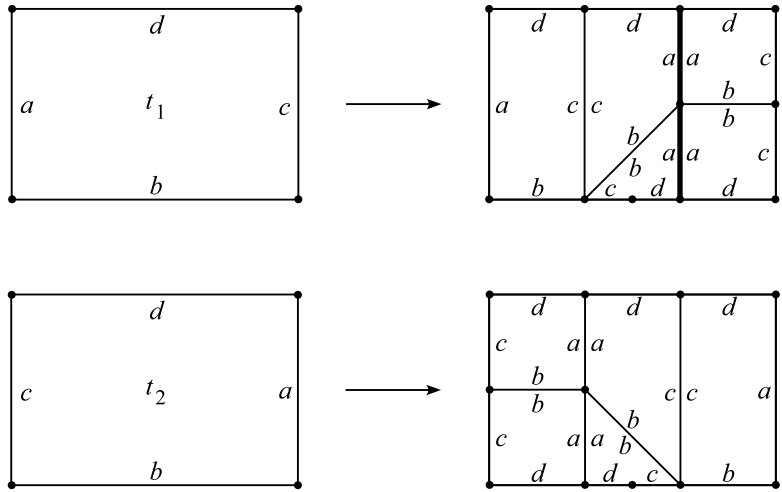}}
\caption{A finite subdivision rule which is not graph hyperbolic}
\label{fig:nhpbc}
  \end{figure}

In this paragraph we define a sequence $\za_1$, $\za_2$,
$\za_3,\ldots$ of arcs in $t_1$.  Two edges which $f$ maps to edge $a$
in the subdivision of $t_1$are drawn with thick line segments.  Let
$\za_1$ be the arc which is the union of these two thick edges in
$t_1$.  We inductively see that for every positive integer $n$ there
exists an arc $\za_n$ joining the top and bottom of $t_1$ consisting
of $2^n$ edges of $\cR^n(t_1)$.  Each of these $2^n$ edges maps to
edge $a$ under $f^n$.  Each of these arcs is to the right of the
previous one.  The open star $S_n$ of $\za_n$ in $\cR^n(t_1)$ is
combinatorially a rectangle tiled by $2^n$ rows and 2 columns of
squares.

The fat path distance between the endpoints of $\za_n$ in $S_n$ is
$2^n+1$.  However, a fat path geodesic joining the endpoints of
$\za_n$ in $\cR^n(t_1)$ might leave $S_n$, and the fat path distance
in $\cR^n(t_1)$ between these points might be less than $2^n+1$.  So
let $m$ be a nonnegative integer, and consider $\cR^m(S_n)$.  Using
the fact that the tiles in both $\cR(t_1)$ and $\cR(t_2)$ can be
organized into three columns, we see that the tiles of $\cR^m(S_n)$
can be organized into $2\cdot 3^m$ columns.  A fat path in
$\cR^{n+m}(t_1)$ starting at an endpoint of $\za_n$ must traverse at
least $3^m$ tiles to leave $\cR^m(S_n)$.  It follows that if $m$ is
large enough, then the fat path distance between the endpoints of
$\za_n$ in $\cR^{n+m}(t_1)$ is exactly $2^n+1$.  The same holds for
$\cR^{n+m}(t_2)$.  So the fat path distance between these points in
$\cR^{n+m}(S^2)$ is independent of $m$ and greater than $2^n$ if $m$
is large enough.  Thus $\cR$ is not graph hyperbolic, since the fat
path subdivision graph will contain arbitrarily large embedded
geodesic square grids.
\end{ex}

\noindent{\bf Remarks:}
\begin{enumerate}

\item Example 7.5 may be generalized by thinking of it as a local
obstruction. If the fsr $\cR$ of this example is a sub-fsr of another
fsr $\cR'$, then the same argument shows that the subdivision complex
for $\cR'$ will not be hyperbolic either.

\item The vertical curve joining sides $b, d$ of Example 7.5 forms a
Thurston obstruction whose corresponding eigenvalue is $1+1/2+1/2=2$.
Obstructions to hyperbolicity are not detectable just by looking at
such eigenvalues. The map $(x,y) \mapsto (2x, y)$ descends to an
affine map which is the subdivision map of a subdivision rule
$\cR'$ whose tiles are images of the usual integral unit squares under
the natural projection. This map also has a Thurston obstruction with
eigenvalue $2$. However, $\cR'$ is contracting and the fat path
subdivision complex is hyperbolic.

\end{enumerate}

\section{The selfsimilarity complex}
\label{sec:selfsim}\nosubsections

This section is devoted to relating the fat path subdivision graph
$\zG$ of $\cR$ with a selfsimilarity complex $\zS$ of $f$.  We recall
the definition of $\zS$.

As in Section~\ref{sec:endo}, we choose a basepoint $p\in S^2\setminus
P_f$.  Let $A$, which we view as an alphabet, be a finite set with the
same cardinality as $f^{-1}(p)$.  For every nonnegative integer $n$,
let $A^n$ be the set of words of length $n$ in $A$, where
$A^0=\{\emptyset\}$.  Set $A^*=\cup _nA^n$.

We next define a bijection $\zL\co A^*\to \cup _n(f^{-n}(p)\times
\{n\})$ so that $\zL(A^n)=f^{-n}(p)\times \{n\}$ for every nonnegative
integer $n$.  The action of $\zL$ on $A^0$ is forced.  Define $\zL$ on
$A=A^1$ to be any bijection to $f^{-1}(p)\times \{1\}$.  For each
$x\in A$, let $\zl_x$ be an arc in $S^2\setminus P_f$ from $p$ to the
first component of $\zL(x)$.  Now let $n$ be a positive integer for
which $\zL$ is defined on $A^n$.  Let $x\in A$ and $w\in A^n$.  Then
$\zL(xw)=(q,n+1)\in f^{-(n+1)}(p)\times \{n+1\}$, where $q$ is the
terminal endpoint of the $f^n$-lift of $\zl_x$ whose initial endpoint
is $\zL(w)$.  This defines $\zL$.

Let $G_p$ be the orbifold fundamental group of $f$ as in
Section~\ref{sec:endo}.  Let $g\in G_p$\,, and let $w\in A^n$ for some
nonnegative integer $n$.  We define $g.w$ as follows.  We choose a
loop $\zg$ in $S^2\setminus P_f$ at $p$ representing $g$.  Then $g.w$
is the element of $A^n$ such that the first component of $\zL(g.w)$ is
the terminal endpoint of the $f^n$-lift of $\zg$ whose initial
endpoint is the first component of $\zL(w)$.  Let $S$ be a generating
set for $G_p$.

Now we define the selfsimilarity complex $\zS$.  It is a graph whose
vertex set is $A^*$.  Given $w\in A^*$, a horizontal edge joins $w$
and $s.w$ for every $s\in S$.  If we have $x\in A$ in addition to
$w\in A^*$, then a vertical edge joins $w$ and $xw$.  We define a
metric on $\zS$ in the straightforward way so that every edge has
length 1.  This defines $\zS$.  Different choices of $p$, $S$ and arcs
$\zl_x$ obtain quasi-isometric graphs.  Let $\zS_n$ denote the
subgraph of $\zS$ spanned by $A^n$ for every nonnegative integer $n$.

We discuss Figure~\ref{fig:selfsim} in this paragraph.  We prefer to
work with special choices for $p$ and $S$.  Let $t$ be the tile of
$\cR^{-1}(S^2)$ as in Section~\ref{sec:defns}.  Let $p$ be a point in
the interior of $t$, and let $S$ be the generating set of $G_p$
determined by $t$ as in Section~\ref{sec:defns}.
Figure~\ref{fig:selfsim} gives a schematic diagram of a portion of the
resulting selfsimilarity complex $\zS$.  The tile $t$ is represented
by a parallelogram at the bottom, disregarding identification of
edges.  The arcs $\zl_x$ are not drawn.  In Figure~\ref{fig:selfsim},
the map $f$ has degree 4 and $t$ lifts to four parallelograms.  The
tilings of $S^2$ obtained by pulling back $t$ under $f$ and $f^2$ need
not be subdivisions of $t$.  These tilings are drawn with gray line
segments, and the edges of $\zS$ are drawn with black line segments.
A prominent property of $\zS$ illustrated by Figure~\ref{fig:selfsim}
is that edges of $\zS_n$ correspond exactly to pairs of lifts of $t$
via $f^n$ which have an edge in common.  So $\zS_n$ is the graph dual
to the 1-skeleton of $f^{-n}(t)$ for every integer $n$.

  \begin{figure} \centerline{\includegraphics{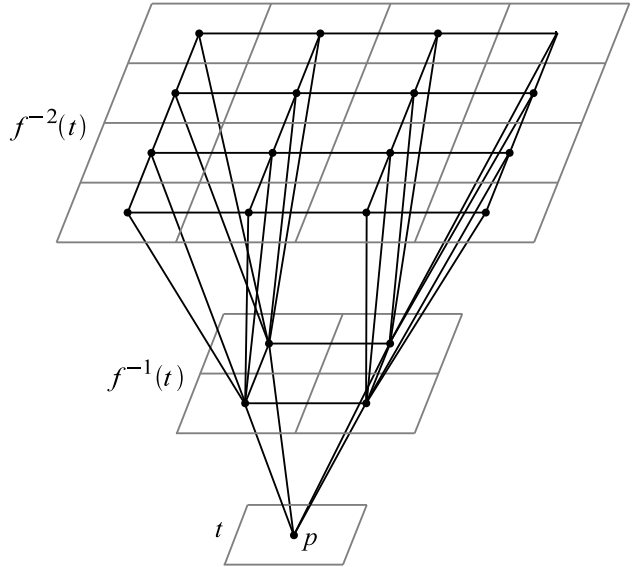}}
\caption{The basepoint $p$ in the tile $t$ of $\cR^{-1}(S^2)$, the
tilings of $S^2$ which pullback $t$ under $f$ and $f^2$ and a portion
of the selfsimilarity complex $\zS$}
\label{fig:selfsim}
  \end{figure}

Here is our result which relates selfsimilarity complexes and
subdivision graphs.

\begin{thm}\label{thm:selfsim} Let $\cR$ be a finite subdivision
rule whose subdivision map $f$ is a Thurston map.  Then every
selfsimilarity complex of $f$ is quasi-isometric to the fat path
subdivision graph of $\cR$.
\end{thm}
  \begin{proof} We define a selfsimilarity complex $\zS$ of $f$ as
above with the following choices.  Let $t$ be the tile of
$\cR^{-1}(S^2)$.  Let $p$ be a point of $S^2$ not in the 1-skeleton of
$\cR^0(S^2)$.  Hence $p$ is in the interior of $t$.  Let $S$ be the
generating set of $G_p$ determined by $t$.  Because $p$ is not in the
1-skeleton of $\cR^0(S^2)$, neither are the elements of $f^{-1}(p)$.
So we may choose arcs $\zl_x$ to lie in the interior of $t$.  To prove
the theorem, it suffices to prove that the fat path subdivision graph
$\zG$ of $\cR$ is quasi-isometric to $\zS$.  In turn, it suffices to
prove that the set $V^*(\zG)$ of vertices of
$\bigcup_{n=0}^{\infty}\zG_n$ is quasi-isometric to the set $V(\zS)$
of vertices of $\zS$.

The vertex set $V(\zS)$ of $\zS$ comes with a bijection from it to
$\bigcup_{n=0}^{\infty}(f^{-n}(p)\times \{n\})$.  Every element of
$f^{-n}(p)$ is contained in a unique lift of $t$ via $f^n$ for every
nonnegative integer $n$.  Thus for every element $v\in V(\zS)$, we
have a tile $\zt(v)$, which is a lift of $t$ via $f^n$ for some
nonnegative integer $n$.  Similarly, for every $v\in V^*(\zG)$, there
exists a tile $\zt(v)$ in some subdivision of $\cR^0(S^2)$.

Let $d_\zG$ and $d_\zS$ be the metrics on $\zG$ and $\zS$,
respectively.  Let $K$ be the number of tiles in $\cR^0(S^2)$, and let
$L$ be the number of tiles in $\cR^1(S^2)$.  

We define a map $\zv\co V^*(\zG)\to V(\zS)$ as follows.  Let $v\in
V^*(\zG)$.  Let $s=\zt(v)$, a tile of $\cR^n(S^2)$ for some
nonnegative integer $n$.  The interior of the tile $f^n(s)$ of
$\cR^0(S^2)$ is contained in the interior of $t$.  Hence $s$ is
contained in some lift $\widetilde{t}$ of $t$ via $f^n$.  There exists
a unique vertex $w\in \zS_n$ such that $\zt(w)=\widetilde{t}$.  We set
$\zv(v)=w$.  This defines $\zv$.

Now we begin to verify that $\zv$ is a quasi-isometry.  Let
$v_1,v_2\in V^*(\zG)$ be the endpoints of a horizontal edge of $\zG$.
Then $s_1=\zt(v_1)$ and $s_2=\zt(v_2)$ are tiles of $\cR^n(S^2)$ for
some nonnegative integer $n$.  There exist lifts $t_1$ and $t_2$ of
$t$ via $f^n$ such that $s_1\subseteq t_1$ and $s_1\subseteq t_2$.
Because $v_1$ and $v_2$ are the endpoints of a horizontal edge,
$s_1\cap s_2$ contains an edge of $\cR^n(S^2)$.  So $t_1\cap t_2$
contains an edge of $f^{-n}(t)$.  Thus the vertices $\zv(v_1)$ and
$\zv(v_2)$ of $\zS$, where $\zt(\zv(v_1))=t_1$ and
$\zt(\zv(v_2))=t_2$, are either equal or they are the endpoints of an
edge.  This proves that if $v_1$ and $v_2$ are the endpoints of a
horizontal edge of $\zG$, then $d_\zS(\zv(v_1),\zv(v_2))\le 1$.

Next suppose that $v_1,v_2\in V^*(\zG)$ are the endpoints of a
vertical edge of $\zG$.  Then after interchanging $v_1$ and $v_2$ if
necessary, we have that $s_1=\zt(v_1)$ is a tile of $\cR^{n-1}(S^2)$
and $s_2=\zt(v_2)$ is a tile of $\cR^n(S^2)$ for some positive integer
$n$.  Moreover, $s_2\subseteq s_1$.  Also, $t_1=\zt(\zv(v_1))$ is the
lift of $t$ via $f^{n-1}$ such that $s_1\subseteq t_1$ and
$t_2=\zt(\zv(v_2))$ is the lift of $t$ via $f^n$ such that
$s_2\subseteq t_2$.  Set Figure~\ref{fig:rectanglesa}.  Now choose any
arc $\zl_x$ in the definition of $\zS$, and let $\widetilde{\zl}_x$ be
the lift of $\zl_x$ to $t_1$ via $f^{n-1}$.  Let $s_3$ be a tile of
$\cR^n(S^2)$ which contains the terminal endpoint of
$\widetilde{\zl}_x$.  Let $t_3$ be the lift of $t$ via $f^n$ such that
$s_3\subseteq t_3$.  Because $t$ contains $L$ tiles of $\cR^1(S^2)$,
the tile $t_1$ contains $L$ tiles of $\cR^n(S^2)$.  Each of these
tiles is contained in a unique lift of $t$ via $f^n$.  Hence $t_1$ is
covered by at most $L$ lifts of $t$ via $f^n$.  Since $s_2\subseteq
s_1\subseteq t_1$, $s_2\subseteq t_2$, $s_3\subseteq t_1$ and
$s_3\subseteq t_3$, the vertices $w_3$ and $w_2$ of $\zS_n$ such that
$\zt(w_3)=t_3$ and $\zt(w_2)=t_2$ are connected by a horizontal edge
path which contains at most $L$ vertices.  The vertices $w_1$ and
$w_3$, where $\zt(w_1)=t_1$, are the endpoints of a vertical edge of
$\zS$.  So
  \begin{equation*}
d_\zS(\zv(v_1),\zv(v_2))  =d_\zS(w_1,w_2)\le
d_\zS(w_1,w_3)+d_\zS(w_3,w_2)\le L.
  \end{equation*}

  \begin{figure}
\centerline{\includegraphics{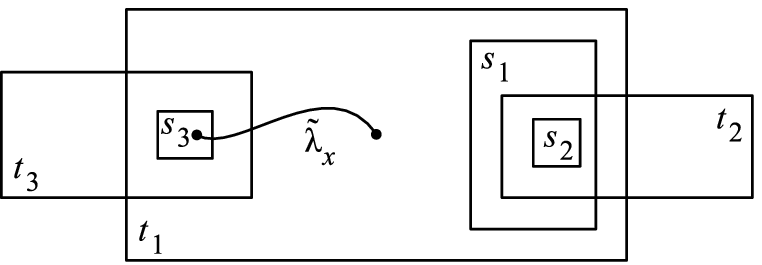}}  \caption{Proving that
$d_\zS(\zv(v_1),\zv(v_2))\le L$}
\label{fig:rectanglesa}
  \end{figure}

We have proved that if $v_1,v_2\in V^*(\zG)$ are the endpoints of an
edge, then $d_\zS(\zv(v_1),\zv(v_2))\le L=Ld_\zG(v_1,v_2)$.  It follows
that if $u,v\in V^*(\zG)$, then $d_\zS(\zv(u),\zv(v))\le
Ld_\zG(u,v)$.  This is one half of what is needed to prove that $\zv$
is a quasi-isometry.

To begin the other half of the proof, let $v_1\in V(\zS)$.  We
consider the set $\zv^{-1}(v_1)$.  The tile $t_1=\zt(v_1)$ is a lift
of $t$ via $f^n$ for some nonnegative integer $n$.  By definition,
$\zv^{-1}(v_1)$ consists of all vertices $v$ of $\zG_n$ such that
$\zt(v)$ is a tile of $\cR^n(S^2)$ which is contained in $t_1$.  In
particular, $\zv$ is surjective.  Moreover, the $d_\zG$-distance
between such vertices of $\zG_n$ is at most $K$.

Now let $v_1,v_2\in V^*(\zG)$ such that $\zv(v_1)$ and $\zv(v_2)$ are
the endpoints of a horizontal edge of $\zS$.  Let $n$ be the
nonnegative integer and let $t_1$ and $t_2$ be the lifts of $t$ via
$f^n$ such that $t_1=\zt(\zv(v_1))$ and $t_2=\zt(\zv(v_2))$.  Then
$s_1=\zt(v_1)$ and $s_2=\zt(v_2)$ are tiles of $\cR^n(S^2)$ such that
$s_1\subseteq t_1$ and $s_2\subseteq t_2$.  Because $\zv(v_1)$ and
$\zv(v_2)$ are the endpoints of a horizontal edge, $t_1\cap t_2$
contains an edge of $f^{-n}(t)$.  It easily follows that
$d_\zG(v_1,v_2)\le 2K\le K+L$.

Finally, let $v_1,v_2\in V^*(\zG)$ such that $\zv(v_1)$ and $\zv(v_2)$
are the endpoints of a vertical edge of $\zS$.  Then after
interchanging $v_1$ and $v_2$ if necessary, we have that
$s_1=\zt(v_1)$ is a tile of $\cR^{n-1}(S^2)$ and $s_2=\zt(v_2)$ is a
tile of $\cR^n(S^2)$ for some positive integer $n$.  Let
$t_1=\zt(\zv(v_1))$ and $t_2=\zt(\zv(v_2))$.  Then $s_1\subseteq t_1$
and $s_2\subseteq t_2$.  See Figure~\ref{fig:rectanglesb}.  Because
$\zv(v_1)$ and $\zv(v_2)$ are the endpoints of a vertical edge of
$\zS$, there exists an arc $\zl_x$ whose lift $\widetilde{\zl}_x$ to
$t_1$ via $f^{n-1}$ has an endpoint in $t_2$.  Let $s_3$ be a tile of
$\cR^n(S^2)$ which contains this endpoint of $\widetilde{\zl}_x$, and
let $v_3$ be the vertex of $\zG_n$ such that $\zt(v_3)=s_3$.  Let
$s_4$ be a tile of $\cR^n(S^2)$ contained in $s_1$, and let $v_4$ be
the vertex of $\zG_n$ such that $\zt(v_4)=s_4$.  Then $v_1$ and $v_4$
are the vertices of a vertical edge of $\zG$, and so
$d_\zG(v_1,v_4)=1$.  Since $s_4\subseteq t_1$ and $s_3\subseteq t_1$,
as in the argument which involves Figure~\ref{fig:rectanglesa}, we
have that $d_\zG(v_4,v_3)<L$.  Since $s_3\subseteq t_2$ and
$s_2\subseteq t_2$, we have that $d_\zG(v_3,v_2)<K$.  So the triangle
inequality implies that $d_\zG(v_1,v_2)\le K+L$.

  \begin{figure} 
\centerline{\includegraphics{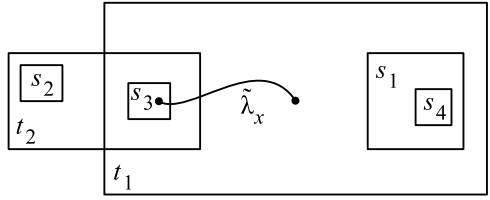}}  \caption{Proving that
$d_\zG(v_1,v_2)\le K+L$}
\label{fig:rectanglesb}
  \end{figure}

Combining the last three paragraphs with the triangle inequality shows
that if $u,v\in V^*(\zG)$, then $d_\zG(u,v)\le
(K+L)d_\zS(\zv(u),\zv(v))+K$.  This completes the proof of
Theorem~\ref{thm:selfsim}.

\end{proof}

\begin{thm}\label{thm:julia} Suppose that $f$ is a rational map which
is the subdivision map of a finite subdivision rule $\cR$.  Then $\cR$
is graph hyperbolic, and the boundary of its fat path subdivision
graph is homeomorphic to the Julia set of $f$.
\end{thm}
  \begin{proof} Theorem~\ref{thm:ratcntrg} implies that $\cR$ is
contracting.  So Theorem~\ref{thm:cntrghpbc} implies that $\cR$ is
graph hyperbolic.  Theorem~\ref{thm:selfsim} implies that the fat path
subdivision graph of $\cR$ is quasi-isometric to the selfsimilarity
complex $\zS$ of $f$.  Finally, Theorem 6.4.4 of Nekrashevych's book
\cite{N} implies that the boundary of $\zS$ is homeomorphic to the
Julia set of $f$.

\end{proof}

\begin{ex}\label{ex:carpet} In this example we construct a Thurston
map which is not rational such that its orbifold fundamental group
virtual endomorphism is contracting and the boundary of its
selfsimilarity complex is a Sierpinski carpet.  We begin with a
general discussion.

Let $\cR$ be a combinatorially expanding finite subdivision rule whose
subdivision map $f$ is a Thurston map.  Statement 3 of
Theorem~\ref{thm:expsep} shows that $\cR$ is weakly isomorphic to a
finite subdivision rule whose mesh approaches 0.  Replacing $\cR$ by a
weakly isomorphic finite subdivision rule does not change the isometry
type of it fat path subdivision graph $\zG$ or the homeomorphism type
of $\partial \zG$.  So we assume that the mesh of $\cR$ approaches 0.
Corollary~\ref{cor:cmblexpn} implies that $\cR$ is contracting.  This
and Theorem~\ref{thm:cntrghpbc} imply that $\zG$ is Gromov hyperbolic.
The discussion after Lemma~\ref{lemma:tilde} describes $\partial \zG$.
So $\partial \zG$ is a subspace of $S^2$ gotten by deleting an open
topological disk about every vertex of $\cR^n(S^2)$ whose valences are
unbounded under subdivision for every nonnegative integer $n$.
Finally, Theorem~\ref{thm:selfsim} shows that the selfsimilarity
complex $\zS$ of $f$ is quasi-isometric to $\zG$.  So the above
describes $\partial \zS$.

We are interested in the case in which $\cR$ does not have bounded
valence.  The assumption of unbounded valence and the fact that the
mesh of $\cR$ approaches 0 imply that $\partial \zG$ and, hence
$\partial \zS$, is the complement in $S^2$ of the union of infinitely
many disjoint open topological disks.  Because the mesh of $\cR$
approaches 0, the diameters of these open disks tend to 0.  By a
theorem of Whyburn \cite{Wh}, any two such spaces are ambiently
homeomorphic by an orientation-preserving homeomorphism of $S^2$.
This means that $\partial \zS$ is a Sierpinski carpet.

Now we construct such a finite subdivision rule $\cR$ and Thurston map
$f$.  Figure~\ref{fig:blowupa} describes $\cR$.  The 1-skeleton of
$\cR^0(S^2)$ is a simple closed curve subdivided into four edges $a$,
$b$, $c$, $d$.  Hence $\cR^0(S^2)$ has two quadrilateral tiles, $t_1$
and $t_2$.  One verifies that Figure~\ref{fig:blowupa} does indeed
describe a finite subdivision rule.  It is easy to see that $\cR$ is
combinatorially expanding.  So Corollary~\ref{cor:cmblexpn} and
Theorem~\ref{thm:cntrn} imply that the orbifold fundamental group
virtual endomorphism of $f$ is contracting.  The four vertices of
$\cR^0(S^2)$ form the postcritical set of $f$, and they all map to the
vertex in the lower left corner of $t_1$.  Moreover, this vertex is a
critical point of $f$.  So its valences are unbounded under
subdivision. It follows that the orbifold of $f$ is hyperbolic.  A
simple closed curve in $S^2\setminus P_f$ which is vertical relative
to the drawing of $t_1$ and $t_2$ in Figure~\ref{fig:blowupa} is a
Thurston obstruction.  So $f$ is not rational.

  \begin{figure}
\centerline{\includegraphics{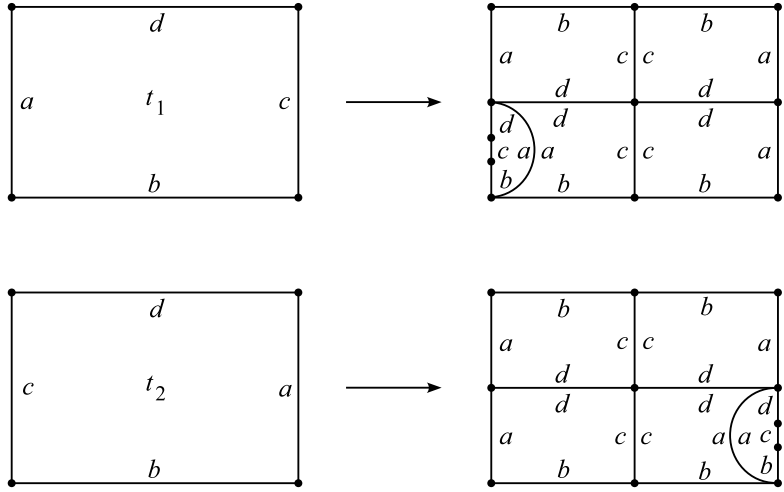}}  \caption{The finite
subdivision rule of Example~\ref{ex:carpet}}
\label{fig:blowupa}
  \end{figure}

In conclusion, $f$ is a Thurston map which is not equivalent to a
rational map, its orbifold fundamental group virtual endomorphism is
contracting and the boundary of its selfsimilarity complex is a
Sierpinski carpet.

\end{ex}

\section{Levy obstructions}\nosubsections
\label{sec:levy}

A \defn{Levy cycle} for a Thurston map $f$ is a multicurve
$\{\gamma_1,\dots,\gamma_k\}$ such that for each $i\in \{1,\dots,k\}$
$\gamma_i$ is homotopic rel $P_f$ to a component of
$f^{-1}(\gamma_{i+1})$ (where the subscripts are modulo $k$) which
maps to $\gamma_{i+1}$ with degree one. In \cite{L}, Levy proves that
if a Thurston map has degree two and has a Thurston obstruction, then
it has a Levy cycle. In \cite{BFH}, Bielefeld, Fisher, and Hubbard
prove that if a topological polynomial has a Thurston obstruction,
then the Thurston obstruction contains a Levy cycle. If a Thurston map
$f$ has a Levy cycle, then its fundamental group virtual endomorphism
$\zf$ is not contracting.

One can generalize the definition of a Levy cycle as follows and still
get an obstruction to the contraction of $\zf$.  A \defn{Levy
obstruction} for a Thurston map $f$ is an essential, nonperipheral,
simple closed curve $\gamma$ in $S^2 \setminus P_f$ for which there
exists a positive integer $n$ such that there is a degree 1 lift of
$\zg$ via $f^n$ to a simple closed curve which is homotopic to $\zg$
rel $P_f$.  Note that every element of a Levy cycle is a Levy
obstruction.

The following theorem motivates our interest in Levy
obstructions.

\begin{thm}\label{thm:levy} Let $f\co S^2\to S^2$ be a Thurston
map with orbifold fundamental group virtual endomorphism $\zf$.  If
$f$ admits a Levy obstruction, then $\zf$ is not contracting.
\end{thm}
  \begin{proof} Suppose that $f$ admits a Levy obstruction.  Then
there exists a simple closed curve $\zg$ in $S^2\setminus P_f$ such
that $\zg$ is neither inessential nor peripheral and there exists a
positive integer $n$ such that there is a degree 1 lift of $\zg$ via
$f^n$ to a simple closed curve $\zd$ which is homotopic to $\zg$ rel
$P_f$.  To prove that $\zf$ is not contracting, it suffices to prove
that the orbifold fundamental group virtual endomorphism of $f^n$ is
not contracting.  This is what we do.  In effect, we replace $f$ by
$f^n$, so that $f(\zd)=\zg$.  We choose $p\in \zg$, $q\in \zd$ with
$f(q)=p$ and a path $\zb$ from $p$ to $q$, the data used to define
$\zf$.  

Let $X$ be the universal covering space of $S^2\setminus P_f$.  We
lift to $X$.  Let $\widetilde{p}$ be a lift of $p$.  Let
$\widetilde{\zg}$ and $\widetilde{\zb}$ be lifts of $\zg$ and $\zb$ to
paths based at $\widetilde{p}$.  Then $\widetilde{\zb}$ ends at a
point $\widetilde{q}$ which lifts $q$.  Let $\widetilde{\zd}$ be a
lift of $\zd$ based at $\widetilde{q}$.

A homotopy $H\co [0,1]\times [0,1]\to S^2\setminus P_f$ from $\zd$ to
$\zg$ also lifts to a homotopy $\widetilde{H}$ from $\widetilde{\zd}$
to $\widetilde{\zg}$.  The homotopy $H$ takes $\zd$ to $\zg$ through a
family of closed curves.  If $\widetilde{\za}$ is a lift of such a
closed curve $\za$, then there exists a deck transformation which
takes the initial endpoint of $\widetilde{\za}$ to its terminal
endpoint.  Because the universal covering map is a local
homeomorphism, these deck transformations are equal for sufficiently
close curves.  This and a compactness argument imply that there exists
a deck transformation $\zs$ such that $\zs(\widetilde{p})$ is the
terminal endpoint of $\widetilde{\zg}$ and $\zs(\widetilde{q})$ is the
terminal endpoint of $\widetilde{\zd}$.  So $\zs\circ \widetilde{\zb}$
is a lift of $\zb$ from the terminal endpoint of $\widetilde{\zg}$ to
the terminal endpoint of $\widetilde{\zd}$.  So
$\widetilde{\zb}\widetilde{\zd}(\zs\circ
\widetilde{\zb})^{-1}\widetilde{\zg}^{-1}$ lifts $\zb \zd
\zb^{-1}\zg^{-1}$ and shows that $\zb \zd \zb^{-1}\zg^{-1}$ is
homotopic to the constant path at $p$ in $S^2\setminus P_f$.

Hence if $g\in G_p$ is the element represented by $\zg$, then $g\in
\text{dom}(\zf)$ and $\zf(g)=g$.  So $g\in \text{dom}(\zf^n)$ and
$\zf^n(g)=g$ for every positive integer $n$.  Since $\zg$ is neither
inessential nor peripheral, $g$ has infinite order in $G_p$.  So the
lengths of the powers of $g$ become arbitrarily large relative to any
finite generating set of $G_p$.  Since they are all fixed by all
powers of $\zf^n$, it follows that $\zf$ is not contracting.

\end{proof}

Consider the finite subdivision rule $\cR_3$ of
Example~\ref{ex:quad15}.  It was observed there that a horizontal
curve gives a Thurston obstruction.  This Thurston obstruction is even
a Levy obstruction.  So Theorems~\ref{thm:cntrn} and \ref{thm:levy}
imply that $\cR_3$ is not contracting.  On the other hand, it is not
hard to verify that $\cR_3$ is graph hyperbolic.  In particular, if a
curve in $S^2\setminus P_f$ is a taut $\zd_m$-geodesic for some
nonnegative integer $m$ with more than 4 segments, then it cannot have
equal level $m$ and level $m+1$ decompositions.

A simpler example is shown in Figure~\ref{fig:blowupb}.  In this case
a vertical curve gives a Levy obstruction.

  \begin{figure}
\centerline{\includegraphics{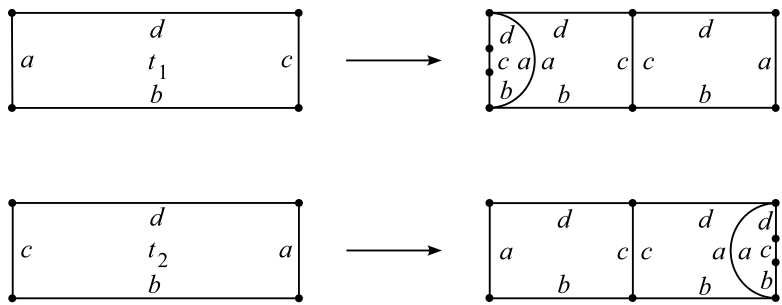}}  \caption{A finite
subdivision rule which is graph hyperbolic but not contracting.}
\label{fig:blowupb}
  \end{figure}

Now suppose that $\cR$ is a finite subdivision rule whose subdivision
map is a Thurston map.  In practice it is easy to check whether $\cR$
is contracting or graph hyperbolic if only because the finite
subdivision rules one encounters in practice are so simple.  However, the 
existence and effectiveness of general algorithms is as yet unclear. 

For general Thurston maps $f$, systematic enumeration of curves will
detect a Levy obstruction, if one exists.  Also, systematic
enumeration of candidate nuclei for associated wreath recursions will
detect if $\phi_f$ is contracting \cite{N}.  But as of this writing,
we do not know if absence of Levy obstructions (and still assuming the
map is not one of the exceptional Latt\`es examples) is sufficient to
prove $\phi_f$ is contracting. In the present context of subdivision
maps of finite subdivision rules, we may ask similar questions.

\begin{enumerate}
  \item Is there an effective algorithm which begins with a finite
subdivision rule $\cR$ associated to a Thurston map and determines
whether or not it is contracting?
  \item Is there an effective algorithm which begins with a finite
subdivision rule $\cR$ associated to a Thurston map and determines
whether or not it is graph hyperbolic?
  \item Is there a Levy obstruction if $\cR$ is not contracting?
  \item Is there a Levy obstruction if $\cR$ is not graph hyperbolic?
\end{enumerate}


\begin{thebibliography}{12}

\bibitem{BFH}
B.~Bielefeld, Y.~Fisher, and J.~Hubbard,
\emph{The classification of critically preperiodic polynomials
as dynamical systems}, J. Amer. Math. Soc. \textbf{5} (1992),
721--762.

\bibitem{BK}
M.~Bonk and B.~Kleiner, \emph{Quasisymmetric parametrizations of
two-dimensional metric spheres}, Invent. Math. {\bf 150}, 1 (2002),
127-183.

\bibitem{BM}
M.~Bonk and D.~Meyer, \emph{Expanding Thurston maps},
http://arxiv.org/abs/1009.3647.

\bibitem{cannon:rmt}
J.~W.~Cannon.  \emph{The combinatorial {R}iemann mapping theorem},  
Acta Math. {\bf 173} (1994), 155-234.  

\bibitem{fsr}
J.~W.~Cannon, W.~J.~Floyd, and W.~R.~Parry, \emph{Finite subdivision
rules}, Conform. Geom. Dyn. \textbf{5} (2001), 153--196
(electronic).

\bibitem{ratsub}
J.~W.~Cannon, W.~J.~Floyd, R.~Kenyon, and W.~R.~Parry,
\emph{Constructing rational maps from subdivision rules}, Conform.
Geom. Dyn. \textbf{7} (2003), 76--102 (electronic).

\bibitem{expi}
J.~W.~Cannon, W.~J.~Floyd, and W.~R.~Parry, \emph{Expansion
complexes for finite subdivision rules I}, Conform. Geom. Dyn.
\textbf{10} (2006), 63--99 (electronic).

\bibitem{expii}
J.~W.~Cannon, W.~J.~Floyd, and W.~R.~Parry, \emph{Expansion
complexes for finite subdivision rules II}, Conform. Geom. Dyn.
\textbf{10} (2006), 326--354 (electronic).

\bibitem{subrat}
J.~W.~Cannon, W.~J.~Floyd, and W.~R.~Parry, \emph{Constructing
subdivision rules from rational maps}, Conform. Geom. Dyn.
\textbf{11} (2007), 128--136 (electronic).

\bibitem{downunder}
J.~W.~Cannon, W.~J.~Floyd, and W.~R.~Parry, \emph{Conformal modulus:
the graph paper invariant or the conformal shape of an algorithm}.  In
Geometric Group Theory Down Under (Canberra, 1996), 71--102, de Gruyter,
Berlin, 1999.

\bibitem{virt}
J.~W.~Cannon, W.~J.~Floyd, W.~R.~Parry, and K.~M.~Pilgrim 
\emph{Subdivision rules and virtual endomorphisms}, Geom. Dedicata
\textbf{141} (2009), 181--195.

\bibitem{CS}
J.~W.~Cannon and E.~L.~Swenson,
{\em Recognizing constant curvature discrete groups in dimension 3},
Trans. Amer. Math. Soc. \textbf{350} (1998), 809--849.

\bibitem{DH}
A.~Douady and J.~H.~Hubbard,
\emph{A proof of Thurston's topological characterization of
rational functions},
Acta Math. {\bf 171} (1993), 263--297.

\bibitem{fpp:exppropii}
W.~J.~Floyd, W.~R.~Parry, and K.~M.~Pilgrim 
\emph{Expansion properties for subdivision rules II: 
obstructions and automata}. Manuscript in preparation.

\bibitem{HP1}
P.~Ha\"issinsky and K.~Pilgrim,
\emph{Coarse expanding conformal dynamics},
Ast\'erisque {\bf 325} (2009).

\bibitem{HP2}
P.~Ha\"issinsky and K.~Pilgrim,
\emph{An algebraic characterization of expanding Thurston maps},
J. Mod. Dyn. {\bf 6} (2012), no. 4, 451–-476. 

\bibitem{L}
S.~Levy,
\emph{Critically finite rational maps},
Ph.D. thesis, Princeton University, 1985.

\bibitem{LS}
Roger C. Lyndon and Paul E. Schupp, {\em Combinatorial Group Theory},
Ergeb. Math. Grenzgeb. \textbf{89}, Springer-Verlag, 1977.

\bibitem{MT}
J.~Milnor and W.~Thurston,
\emph{On iterated maps of the interval}, 
in Dynamical systems, Springer Verlag Lecture Notes 1342, 1988.  

\bibitem{N}
Volodymyr Nekrashevych, {\em Self-Similar Groups},
Math. Surveys and Monographs \textbf{117},
Amer. Math. Soc., Providence, 2005.

\bibitem{No}
M.~Noonan,
\emph{FractalStream},
computer software written by M. Noonan, available from
https://code.google.com/p/fractalstream/.

\bibitem{kmp:cds}
K.~M.~Pilgrim,
\emph{Combinations of complex dynamical systems}.
Springer-Verlag Lecture Notes in Mathematics {\bf 1827},
Berlin, 2003. 

\bibitem{R}
B.~Rushton, \emph{Classification of subdivision rules for geometric
groups of low dimension}, Conform. Geom. Dyn.\textbf{18} (2014),
171--191 (electronic).

\bibitem{S}
M.~Shub,
\emph{Endomorphisms of compact differentiable manifolds},
Amer. J. Math. {\bf 91} (1969), 175--199.

\bibitem{CP}
K. Stephenson,
\emph{CirclePack}, software, available from
http://www.math.utk.edu/\~{}kens.

\bibitem{Th}
W.~P.~Thurston, 
\emph{On the geometry and dynamics of iterated rational maps},
edited by D. Schleicher and N. Selinger and with an appendix by
Schleicher. In Complex Dynamics, 
3--137, A K Peters, Wellesley, MA,  2009. 

\bibitem{Wh}
G.~T.~Whyburn,
\emph{Topological characterization of the {S}ierpi\'nski curve},
Fund. Math. {\bf 45} (1958), 320-324.

\end{thebibliography}
\end{document}